\documentclass[11pt,a4paper]{article}
\usepackage{mathrsfs}
\usepackage{bbding}
\usepackage{amsfonts}
\usepackage{epsf,epsfig,amsfonts,amsgen,amsmath,amstext,amsbsy,amsopn,amsthm,lineno}
\usepackage{color}
\newtheorem{theorem}{Theorem}

\newtheorem{lemma}{Lemma}
\newtheorem{false statement}{False Statement}
\newtheorem{corollary}{Corollary}

\theoremstyle{definition}

\newtheorem{claim}{Claim}

\newtheorem{case}{Case}
\newtheorem{subcase}{Case}[case]
\newtheorem{subsubcase}{Case}[subcase]
\newtheorem{subsubsubcase}{Case}[subsubcase]

\newcounter{mathitem}
\newenvironment{mathitem}
  {\begin{list}{{$(\roman{mathitem})$}}{
   \setcounter{mathitem}{0}
   \usecounter{mathitem}
   \setlength{\topsep}{0pt plus 2pt minus 0pt}
   \setlength{\parskip}{0pt plus 2pt minus 0pt}
   \setlength{\partopsep}{0pt plus 2pt minus 0pt}
   \setlength{\parsep}{0pt plus 2pt minus 0pt}
   \setlength{\leftmargin}{20pt}
   \setlength{\itemsep}{0pt plus 2pt minus 0pt}}}
  {\end{list}}

\newcommand{\pa}{{\rm par}}

\begin{document}

\title{\bf\Large The Ramsey numbers of paths versus wheels: a complete
solution\thanks{Supported by NSFC (No. 11271300) and the Doctorate
Foundation of Northwestern Polytechnical University (No. cx201202
and No. cx201326). E-mail addresses: libinlong@mail.nwpu.edu.cn (B.
Li), ningbo\_math84@mail.nwpu.edu.cn (B. Ning).}}

\date{}

\author{Binlong Li$^{a,b}$ and Bo Ning$^a$\\[2mm]
\small $^a$ Department of Applied Mathematics,
\small Northwestern Polytechnical University,\\
\small Xi'an, Shaanxi 710072, P.R.~China\\
\small $^b$ Department of Mathematics, University of West Bohemia,\\
\small Univerzitn\'i 8, 306 14 Plze\v n, Czech Republic\\}
\maketitle

\begin{abstract}
Let $G_1$ and $G_2$ be two given graphs. The Ramsey number
$R(G_1,G_2)$ is the least integer $r$ such that for every graph $G$
on $r$ vertices, either $G$ contains a $G_1$ or $\overline{G}$
contains a $G_2$. We denote by $P_n$ the path on $n$ vertices and
$W_m$ the wheel on $m+1$ vertices. Chen et al. and Zhang determined
the values of $R(P_n,W_m)$ when $m\leq n+1$ and when $n+2\leq m\leq
2n$, respectively. In this paper we determine all the values of
$R(P_n,W_m)$ for the left case $m\geq 2n+1$. Together with Chen et
al's and Zhang's results, we give a complete solution to the problem
of determining the Ramsey numbers of paths versus wheels.

\smallskip


\end{abstract}

\section{Introduction}

We use Bondy and Murty \cite{Bondy_Murty} for terminology and
notation not defined here, and consider finite simple graphs only.

Let $G$ be a graph. We denote by $\nu(G)$ the order of $G$, by
$\delta(G)$ the minimum degree of $G$, and by $\omega(G)$ the
component number of $G$. We denote by $P_n$ and $C_n$ the path and
cycle on $n$ vertices, respectively. The \emph{wheel} on $n+1$
vertices, denoted by $W_n$, is the graph obtained by joining a
vertex to each vertex of a $C_n$.

Let $G_1$ and $G_2$ be two graphs. The \emph{Ramsey number}
$R(G_1,G_2)$, is defined as the least integer $r$ such that for
every graph $G$ on $r$ vertices, either $G$ contains a $G_1$ or
$\overline{G}$ contains a $G_2$, where $\overline{G}$ is the
complement of $G$. If $G_1$ and $G_2$ are both complete, then
$R(G_1,G_2)$ is the classical Ramsey number $r(\nu(G_1),\nu(G_2))$.
Otherwise, $R(G_1,G_2)$ is usually called the \emph{generalized
Ramsey number}.

In 1967, Gerencs\'er and Gy\'arf\'as \cite{Gerencser_Gyarfas}
computed the Ramsey numbers of all path-path pairs, and gave the
first generalized Ramsey number formula. (In fact, this question of
determining Ramsey numbers of paths versus paths appeared in a paper
of Erd\"{o}s \cite{Erdos} in 1947, and the right upper bound was
also determined there.) After that, Faudree et al.
\cite{Faudree_Lawrence_Parsons_Schelp} determined the Ramsey numbers
of paths versus cycles. We list these results as bellow, both of
them will be used in this paper.

\begin{theorem}[Gerencs\'er and Gy\'arf\'as \cite{Gerencser_Gyarfas}]
If $m\geq n\geq 2$, then
$$R(P_n,P_m)=m+\lfloor n/2\rfloor-1.$$
\end{theorem}

\begin{theorem}[Faudree et al. \cite{Faudree_Lawrence_Parsons_Schelp}]
If $n\geq 2$ and $m\geq 3$, then \\
$$R(P_n,C_m)=\left\{\begin{array}{ll}
2n-1,           & \mbox{for } n\geq m \mbox{ and } m \mbox{ is odd};\\
n+m/2-1,        & \mbox{for } n\geq m \mbox{ and } m \mbox{ is even};\\
\max\{m+\lfloor n/2\rfloor-1,2n-1\},    & \mbox{for } m>n \mbox{ and }
                    m \mbox{ is odd};\\
m+\lfloor n/2\rfloor-1,     & \mbox{for } m>n \mbox{ and } m \mbox{
                    is even}.
\end{array}\right.$$
\end{theorem}

Recently, graph theorists have begun to investigate the Ramsey
numbers of paths versus wheels. Baskoro and Surahmat
\cite{Bakoro_Surahmat} conjectured the values of $R(P_n,W_m)$ when
$n\geq m-1$, and got some partial results. Chen et al.
\cite{Chen_Zhang_Zhang} completely determined the values of
$R(P_n,W_m)$ when $n\geq m-1$. Salman and Broersma
\cite{Salmana_Broersma} further generalized Chen et al.'s result.
Zhang \cite{Zhang} firstly obtained all the values of $R(P_n,W_m)$
when $n+2\leq m\leq 2n$. We list the results of Chen et al.'s and
Zhang's in the following.

\begin{theorem}[Chen et al. \cite{Chen_Zhang_Zhang}]
If $3\leq m\leq n+1$, then
$$R(P_n,W_m)=\left\{\begin{array}{ll}
3n-2,   &m \mbox{ is odd};\\
2n-1,   &m \mbox{ is even}.
\end{array}\right.$$
\end{theorem}

\begin{theorem}[Zhang \cite{Zhang}]
If $n+2\leq m\leq 2n$, then
$$R(P_n,W_m)=\left\{\begin{array}{ll}
3n-2,   &m \mbox{ is odd};\\
m+n-2,  &m \mbox{ is even}.
\end{array}\right.$$
\end{theorem}

For the case $m\geq 2n+1$, some upper bounds and lower bounds of
$R(P_n,W_m)$ were given \cite{Salmana_Broersma,Zhang}. Furthermore,
for some $n,m$, the exact values of $R(P_n,W_m)$ were also
determined in \cite{Salmana_Broersma,Zhang}.

In this paper we will prove the following formula, which can be used
to determine all the values of $R(P_n,W_m)$ for the left case $m\geq
2n+1$.

\begin{theorem}
If $n\geq 2$ and $m\geq 2n+1$, then
$$R(P_n,W_m)=\left\{\begin{array}{ll}
(n-1)\cdot\beta+1,  &\alpha\leq\gamma;\\
\lfloor(m-1)/\beta\rfloor+m,  &\alpha>\gamma,
\end{array}\right.$$
where
$$\alpha=\frac{m-1}{n-1}, \beta=\lceil\alpha\rceil \mbox{ and }
\gamma=\frac{\beta^2}{\beta+1}.$$
\end{theorem}

Together with Theorems 4 and 5, we give a complete solution to the
problem of determining the Ramsey numbers of paths versus wheels.

\section{Preliminaries}

Before our proof we will first list one result due to Zhang
\cite{Zhang} and give some additional terminology and notation.
Second, we will prove a series of lemmas which support our proof of
the main theorem.

The following result is a rewriting of two corollaries in
\cite{Zhang}. It helps us to deal with the cases $n=3,4$ in our
proof.

\begin{theorem}[Zhang \cite{Zhang}]
If $n\geq 3$ and $m\geq 2n+1$, then
$$R(P_n,W_m)=\left\{\begin{array}{ll}
m+n-1,  & \mbox{if } m=1\mod(n-1);\\
m+n-2,  & \mbox{if } m=0,2\mod(n-1).
\end{array}\right.$$
\end{theorem}

For integers $s,t$, the \emph{interval} $[s,t]$ is the set of
integers $i$ with $s\leq i\leq t$. Note that if $s>t$, then
$[s,t]=\emptyset$. Let $X$ be a subset of $\mathbb{N}$. We set
$\mathcal {L}(X)=\{\sum_{i=1}^kx_i: x_i\in X, k\in \mathbb{N}\}$,
and suppose $0\in\mathcal{L}(X)$ for any set $X$. Note that if $1\in
X$, then $\mathcal{L}(X)=\mathbb{N}$. For an interval $[s,t]$, we
use $\mathcal{L}[s,t]$ instead of $\mathcal{L}([s,t])$.

In the following of the paper, $n$ always denotes an integer at
least 2 and $m$ an integer at least 3. We denote by $\pa(n)$ the
parity of $n$, i.e., $\pa(n)=\lceil n/2\rceil-\lfloor n/2\rfloor$.

For integers $n,m$, let $t(n,m)$ be the values of $R(P_n,W_m)$
defined in Theorem 5, that is,
$$t(n,m)=\left\{\begin{array}{ll}
(n-1)\cdot\beta+1,  &\alpha\leq\gamma;\\
\lfloor(m-1)/\beta\rfloor+m,  &\alpha>\gamma,
\end{array}\right.$$
where
$$\alpha=\frac{m-1}{n-1}, \beta=\lceil\alpha\rceil \mbox{ and }
\gamma=\frac{\beta^2}{\beta+1}.$$

\begin{lemma}
  If $m\geq 2n+1$, then $t(n,m)=\min\{t: t\notin \mathcal {L}[t-m+1,n-1]\}$.
\end{lemma}

\begin{proof}
Set $T=\{t: t\in\mathcal{L}[t-m+1,n-1]\}$. Note that if $t\in T$,
then $t-1\in T$. So it is sufficient to prove that
$t(n,m)=\max(T)+1$.

Note that
\begin{align*}
                & t\in T \Leftrightarrow t\in\mathcal{L}[t-m+1,n-1]\\
\Leftrightarrow & t\in[k(t-m+1),k(n-1)], \mbox{ for some integer } k\\
\Leftrightarrow & t\leq\frac{k}{k-1}(m-1) \mbox{ and } t\leq k(n-1),
                    \mbox{ for some integer } k\\
\Leftrightarrow & t\leq k(n-1) \mbox{ for some integer } k<\alpha+1,
                    \mbox{ or}\\
                & t\leq\left\lfloor\frac{m-1}{k-1}\right\rfloor+m-1,
                    \mbox{ for some integer } k\geq\alpha+1.
\end{align*}
This implies that
$$T=\left\{t: t\leq k(n-1), k\leq\beta\right\}\cup\left\{t:
t\leq\left\lfloor\frac{m-1}{k-1}\right\rfloor+m-1,
k\geq\beta+1\right\}.$$
Thus
\begin{align*}
\max(T) & =\max\left\{(n-1)\beta,
            \left\lfloor\frac{m-1}{\beta}\right\rfloor+m-1\right\}\\
        & =\left\{\begin{array}{ll}
            (n-1)\cdot\beta,                &\alpha\leq\gamma;\\
            \lfloor(m-1)/\beta\rfloor+m-1,  &\alpha>\gamma.
            \end{array}\right.
\end{align*}
We conclude that $t(n,m)=\max(T)+1$.
\end{proof}

\begin{lemma}
Let $G$ be a graph on at least three vertices.
\begin{mathitem}
\item[(1)] If $G$ is 2-connected and
$\delta(G)\geq\lceil n/2\rceil$, then $G$ contains a cycle of order
at least $\min\{\nu(G),n\}$.
\item[(2)] If $x\in V(G)$, $G$ is connected and $d(v)\geq n-1$ for
every vertex $v\in V(G)\backslash\{x\}$, then $G$ contains a path
from $x$ of order at least $n$.
\item[(3)] If $x,y\in V(G)$, $G+xy$ is 2-connected and $d(v)\geq n-1$ for
every vertex $v\in V(G)\backslash\{x,y\}$, then $G$ contains a path
from $x$ to $y$ of order at least $n$.
\item[(4)] If $x,y\in V(G)$, $G+xy$ is 2-connected
and $d(v)\geq\lceil n/2\rceil$ for every vertex $v\in
V(G)\backslash\{x,y\}$, then $G$ contains a path from $x$ of order
at least $\min\{\nu(G),n\}$.
\item[(5)] If $G$ is connected and
$\delta(G)\geq\lfloor n/2\rfloor$, then $G$ contains a path of order
at least $\min\{\nu(G),n\}$.
\item[(6)] If $x\in V(G)$, $G$ is connected, and $d_{G-x}(v)\geq
n-2$ for every vertex $v\in V(G)\backslash\{x\}$, then $G$ contains
a path from $x$ of order at least $n$.
\item[(7)] If $x\in V(G)$, $G$ is 2-connected and
$d_{G-x}(v)\geq\lfloor n/2\rfloor$ for every vertex $v\in
V(G)\backslash\{x\}$, then $G$ contains a path from $x$ of order at
least $\min\{\nu(G),n\}$.
\end{mathitem}
\end{lemma}

\begin{proof}
The assertions (1), (2) and (3) are results of Dirac \cite{Dirac},
Erd\"os and Gallai \cite{Erdos_Gallai}, respectively. Now we prove
the other assertions.

(4) Let $G'=G+xy$. Since every two nonadjacent vertices of $G'$
contain one with degree at least $\lceil n/2\rceil$, by Fan's
Theorem \cite{Fan}, $G'$ contains a cycle $C$ with order at least
$\min\{\nu(G),n\}$. If $C$ does not contain the added edge $xy$,
then $C$ is a cycle of $G$ and $G$ contains a path from $x$ of order
at least $\min\{\nu(G),n\}$; if $C$ contains the added edge $xy$,
then $P=C-xy$ is a path of $G$ from $x$ of order at least
$\min\{\nu(G),n\}$.

(5) We add a new vertex $x$ and join $x$ to every vertex of $G$. We
denote the resulting graph as $G'$. Thus every vertex in $V(G')$ has
degree at least $\lfloor n/2\rfloor+1=\lceil(n+1)/2\rceil$. By (1),
$G'$ contains a cycle of order at least $\min\{\nu(G'),n+1\}$, and
$G$ contains a path of order at least $\min\{\nu(G),n\}$.

(6) Let $H$ be a component of $G-x$, and let $x'$ be a neighbor of
$x$ in $H$. Note that every vertex in $H$ has degree at least $n-2$
in $H$. By (2), $H$ contains a path $P_1$ from $x'$ of order at
least $n-1$. Thus $P=xx'P_1$ is a path from $x$ of order at least
$n$ in $G$.

(7) Let $G'=G-x$. If $G'$ contains a vertex with degree 1, then
$n\leq 3$ and the assertion is trivially true. Now we assume that
$\delta(G')\geq 2$.

We first assume that $G'$ is 2-connected. By (1), $G'$ contains a
cycle $C$ of order at least $\min\{\nu(G'),n-\pa(n)\}$. Let $P_1$ be
a path from $x$ to $C$, let $x'$ be the end-vertex of $P_1$ on $C$,
and let $x''$ be a neighbor of $x'$ on $C$. Then $P=P_1\cup C-x'x''$
is a path from $x$ of order at least $\min\{\nu(G),n\}$.

Now we assume that $G'$ is separable. Then every end-block of $G'$
is 2-connected. Let $B$ be an end-block of $G'$, and $b$ be the
cut-vertex of $G'$ contained in $B$. Since $G$ is 2-connected, $x$
is adjacent to some vertex, say $x'$, in $B-b$. By (3), $B$ contains
a path $P_1$ from $x'$ to $b$ of order at least $\lfloor
n/2\rfloor+1$, and by (2), $H-(B-b)$ contains a path $P_2$ from $b$
of order at least $\lfloor n/2\rfloor+1$. Thus $P=xx'P_1bP_2$ is a
path from $x$ of order at least $n$.
\end{proof}

\begin{lemma}
If $G$ is a disconnected graph such that
\begin{mathitem}
\item[(1)] $m\leq \nu(G)$; and
\item[(2)] every component of $G$ has order at most $\lfloor m/2\rfloor$,
\end{mathitem}
then $\overline{G}$ contains a $C_m$.
\end{lemma}

\begin{proof}
Let $G'$ be an induced subgraph of $G$ with order $m$. Clearly every
component of $G'$ has order at most $\lfloor m/2\rfloor$. Thus every
vertex of $G'$ has degree at least $\lceil m/2\rceil$ in
$\overline{G'}$. By Lemma 2, $\overline{G'}$ contains a $C_m$.
\end{proof}

\begin{lemma}
Let $G$ be a graph.
\begin{mathitem}
\item[(1)] If $n\leq\nu(G)\leq\lfloor3n/2\rfloor-2$ and $G$ contains no
$P_n$, then $\overline{G}$ contains a path of order $2\nu(G)+3-2n$.
\item[(2)] If $\nu(G)\geq\lfloor3n/2\rfloor-1$ and $G$ contains no $P_n$,
then $\overline{G}$ contains a path of order $\nu(G)+1-\lfloor
n/2\rfloor$.
\item[(3)] If $n\geq 4$ is even, $\nu(G)\geq3n/2-1$, and $G$ contains no
$C_n$ then $\overline{G}$ contains a path of order $\nu(G)+1-n/2$.
\end{mathitem}
\end{lemma}

\begin{proof}
The lemma can be deduced by Theorems 1 and 2.
\end{proof}

\begin{lemma}
Let $G_1$ and $G_2$ be two disjoint graphs. If
\begin{mathitem}
\item[(1)] $\overline{G_1}$ contains a path of order $p\geq 2$; and
\item[(2)] $m\leq\min\{2\nu(G_1),\nu(G_1)+\nu(G_2),p+2\nu(G_2)-1\}$,
\end{mathitem}
then $\overline{G_1\cup G_2}$ contains a $C_m$.
\end{lemma}

\begin{proof}
We first assume that $\nu(G_2)\geq\lfloor m/2\rfloor$. If $m$ is
even, then $\nu(G_1)\geq m/2$ and $\nu(G_2)\geq m/2$. Let
$x_1,x_2,\ldots,x_k$ be $k=m/2$ vertices in $G_1$, and let
$y_1,y_2,\ldots,y_k$ be $k$ vertices in $G_2$. Then
$C=x_1y_1x_2y_2\cdots x_ky_kx_1$ is a $C_m$ in $\overline{G_1\cup
G_2}$. If $m$ is odd, then then $\nu(G_1)\geq(m+1)/2$ and
$\nu(G_2)\geq(m-1)/2$. Note that $G_1$ has two nonadjacent vertices.
Let $x_1,x_2,\ldots,x_k$ be $k=(m+1)/2$ vertices in $G_1$ such that
$x_1x_k\notin E(G_1)$, and let $y_1,y_2,\ldots,y_{k-1}$ be $k-1$
vertices in $G_2$. Then $C=x_1y_1x_2y_2\cdots x_{k-1}y_{k-1}x_kx_1$
is a $C_m$ in $\overline{G_1\cup G_2}$.

Now we assume that $\nu(G_2)\leq\lfloor m/2\rfloor-1$. Let
$V(G_2)=\{y_1,y_2,\cdots,y_k\}$, where $k=\nu(G_2)$. Since $2\leq
m+1-2k\leq p$, $\overline{G_1}$ contains a path $P$ of order
$m+1-2k$. Let $s,t$ be the two end-vertices of $P$. Note that
$\nu(G_1)-\nu(P)\geq m-k-m-1+2k=k-1$. Let $x_1,x_2,\ldots,x_{k-1}$
be $k-1$ vertices in $V(G_1-P)$. Then $C=sy_1x_1y_2x_2\cdots
x_{k-1}y_ktP$ is a $C_m$ in $\overline{G_1\cup G_2}$.
\end{proof}

\begin{lemma}
Suppose $m\geq 2n+1$. Let $G$ be a disconnected graph containing no
$P_n$. If
\begin{mathitem}
\item[(1)] $m\leq \nu(G)$; and
\item[(2)] the order sum of every $\omega(G)-1$ components in $G$ is at least
$m+\lfloor n/2\rfloor-\nu(G)$,
\end{mathitem}
then $\overline{G}$ contains a $C_m$.
\end{lemma}

\begin{proof}
If every component of $G$ has order at most $\lfloor m/2\rfloor$,
then we are done by Lemma 3. Now we assume that there is a component
$H$ with order at least $\lfloor m/2\rfloor+1$.

Let $G_1=H$, and $G_2=G-H$. Note that $m\leq 2\nu(G_1)$,
$m\leq\nu(G)=\nu(G_1)+\nu(G_2)$ and $\nu(G_2)\geq m+\lfloor
n/2\rfloor-\nu(G)$.

Note that $\nu(G_1)\geq\lfloor m/2\rfloor+1\geq n$. If
$\nu(G_1)\leq\lfloor3n/2\rfloor-2$, then by Lemma 4,
$\overline{G_1}$ contains a path of order $p=2\nu(G_1)+3-2n$. Since
\begin{align*}
        & p+2\nu(G_2)-1=2\nu(G_1)+3-2n+2\nu(G_2)-1\\
=       & 2\nu(G)+2-2n\geq 2m+2-2n\geq m,
\end{align*}
by Lemma 5, $\overline{G}$ contains a $C_m$. If
$\nu(G_1)\geq\lfloor3n/2\rfloor-1$, then by Lemma 4,
$\overline{G_1}$ contains a path of order $p=\nu(G_1)+1-\lfloor
n/2\rfloor$. Since
\begin{align*}
    & p+2\nu(G_2)-1=\nu(G_1)+1-\left\lfloor\frac{n}{2}\right\rfloor+2\nu(G_2)-1\\
=   & \nu(G)+\nu(G_2)-\left\lfloor\frac{n}{2}\right\rfloor
        \geq\nu(G)+m+\left\lfloor\frac{n}{2}\right\rfloor-\nu(G')
        -\left\lfloor\frac{n}{2}\right\rfloor=m,
\end{align*}
by Lemma 5, $\overline{G}$ contains a $C_m$.
\end{proof}

\begin{lemma}
Let $G$ be a graph, $X$ an independent set of $G$, $R=G-X$. If
\begin{mathitem}
\item[(1)] $|X|\geq 3$;
\item[(2)] every component of $R$ is joined to at most one vertex in $X$;
\item[(3)] $\overline{R}$ contains a path of order $p\geq 2$; and
\item[(4)] $m\leq\min\{\nu(G),p+2|X|-3\}$,
\end{mathitem}
then $\overline{G}$ contains a $C_m$.
\end{lemma}

\begin{proof}
Let $P$ be a path in $\overline{R}$ with the largest order. Clearly
$\nu(P)\geq p$.

If $\nu(P)\geq m-1$, then let $P'$ be a subpath of $P$ of order
$m-1$. Let $s,t$ be the two end-vertices of $P'$. Since each of $s$
and $t$ is adjacent to at most one vertex in $X$ and $|X|\geq 3$,
there is a vertex $x$ in $X$ nonadjacent to both $s$ and $t$. Thus
$C=sxtP'$ is a $C_m$ in $\overline{G}$. Now we assume that
$\nu(P)\leq m-2$.

Let $s,t$ be the two end-vertices of $P$. If $P$ contains all
vertices in $R$, then $\nu(P)=\nu(R)$. Let $x$ be a vertex in $X$
nonadjacent to $s$, and $x'$ be a vertex in $X\backslash\{x\}$
nonadjacent to $t$. Note that $|X|=\nu(G)-\nu(R)\geq m-\nu(P)$. Let
$x_1,x_2,\ldots,x_k$ be $k=m-\nu(P)$ vertices in $X$ such that
$x_1=x$ and $x_k=x'$, then $C=sx_1x_2\cdots x_ktP$ is a $C_m$ in
$\overline{G}$. Now we assume that $V(R)\backslash
V(P)\neq\emptyset$.

Let $U=V(R-P)$. Note that each of $s,t$ is adjacent to every vertex
in $U$, and this implies that $U\cup\{s,t\}$ is contained in a
component of $R$. Thus $U\cup\{s,t\}$ is joined to at most one
vertex in $X$. Let $y$ be the vertex in $X$ that is joined to
$U\cup\{s,t\}$. If such a vertex does not exist, then let $y$ be any
one vertex in $X$.

Note that $m-\nu(P)\leq m-p\leq 2|X|-3$. If $m-p$ is odd, then
$|X|\geq(m-\nu(P)+1)/2+1$. Let $x_1,x_2,\ldots,x_k$ be
$k=(m-\nu(P)+1)/2$ vertices in $X\backslash\{y\}$, and let
$u_1,\ldots,u_{k-1}$ be $k-1$ vertices in $U\cup
X\backslash\{x_1,x_2,\ldots,x_k\}$. Then $C=sx_1u_1x_2u_2\cdots
x_{k-1}u_{k-1}x_ktP$ is a $C_m$ in $\overline{G}$. If $m-\nu(P)$ is
even, then $m-\nu(P)\leq 2|X|-4$ and $|X|\geq(m-\nu(P))/2+2$. Let
$x_1,x_2,\ldots,x_k$ be $k=(m-\nu(P))/2+1$ vertices in
$X\backslash\{y\}$, and let $u_1,\ldots,u_{k-2}$ be $k-2$ vertices
in $U\cup X\backslash\{x_1,x_2,\ldots,x_k\}$. Then
$C=sx_1u_1x_2u_2\cdots x_{k-2}u_{k-2}$ $x_{k-1}x_ktP$ is a $C_m$ in
$\overline{G}$.
\end{proof}

\begin{lemma}
Let $G$ be a graph, $X_1,X_2$ two independent sets of $G$ (possibly
joint), $X=X_1\cup X_2$, $R=G-X$. If
\begin{mathitem}
\item[(1)] $|X_1|=|X_2|\geq 3$, $|X_1\backslash
X_2|=|X_2\backslash X_1|\geq 2$;
\item[(2)] every component of $R$ is joined to at most one vertex in $X_i$, $i=1,2$;
\item[(3)] $\overline{R}$ contains a path of order $p\geq 2$; and
\item[(4)] $m\leq\min\{\nu(G),p+2|X|-5\}$,
\end{mathitem}
then $\overline{G}$ contains a $C_m$.
\end{lemma}

\begin{proof}
We first define an \emph{adjustable segment} of a cycle $C$. If
$X_1\cap X_2=\emptyset$, then letting $x_1,x'_1,x''_1\in X_1$,
$x_2,x'_2,x''_2\in X_2$ and $u\in V(R)$, we call a subpath $A$ an
adjustable segment of $C$ with the center $u$ if one of the
following is true:
\begin{mathitem}
\item[(1)] $A=x_1x'_1ux'_2x_2$ with $x''_1,x''_2\notin V(C)$;
\item[(2)] $A=x_1x'_1x''_1ux'_2x_2$ with $x''_2\notin V(C)$;
\item[(3)] $A=x_1x'_1ux''_2x'_2x_2$ with $x''_1\notin V(C)$; or
\item[(4)] $A=x_1x'_1x''_1ux''_2x'_2x_2$.
\end{mathitem}
If $X_1\cap X_2\neq\emptyset$, then letting $x_1,x'_1\in
X_1\backslash X_2$, $x_2,x'_2\in X_2\backslash X_1$ and $x\in
X_1\cap X_2$, we call a subpath $A$ an adjustable segment of $C$
with the center $x$ if one of the following is true:
\begin{mathitem}
\item[(1)] $A=x_1xx_2$ with $x'_1,x'_2\notin V(C)$;
\item[(2)] $A=x_1x'_1xx_2$ with $x'_2\notin V(C)$;
\item[(3)] $A=x_1xx'_2x_2$ with $x'_1\notin V(C)$; or
\item[(4)] $A=x_1x'_1xx'_2x_2$.
\end{mathitem}

If $X_1\cap X_2\neq\emptyset$, then let $P$ be a path in
$\overline{R}$ with the largest order; if $X_1\cap X_2=\emptyset$,
then let $P$ be a non-Hamilton path in $\overline{R}$ with the
largest order.

If $\nu(P)\geq m-5$, then let $P'$ be a subpath of $P$ of order
$m-5$ and $s,t$ be the two end-vertices of $P'$. If $X_1\cap
X_2\neq\emptyset$, then let $x$ be a vertex in $X_1\cap X_2$, $x_1$
a vertex in $X_1\backslash X_2$ nonadjacent to $s$, $x'_1$ a vertex
in $X_1\backslash X_2\backslash\{x_1\}$, $x_2$ a vertex in
$X_2\backslash X_1$ nonadjacent to $t$ and $x'_2$ a vertex in
$X_2\backslash X_1\backslash\{x_2\}$. Then $C=sx_1x'_1xx'_2x_2tP'$
is a $C_m$ in $\overline{G}$. If $X_1\cap X_2=\emptyset$, then let
$u$ be a vertex in $V(R-P')$, $x_1$ a vertex in $X_1$ nonadjacent to
$s$, $x'_1$ a vertex in $X_1\backslash\{x_1\}$ nonadjacent to $u$,
$x_2$ a vertex in $X_2$ nonadjacent to $t$ and $x'_2$ a vertex in
$X_2\backslash\{x_2\}$ nonadjacent to $u$. Then
$C=sx_1x'_1ux'_2x_2tP'$ is a $C_m$ in $\overline{G}$.

Now we assume that $\nu(P)\leq m-6$. By a similar argument in the
analysis above, we can get a cycle $C$ in $\overline{G}$ of order at
least $\nu(P)+5$ such that
\begin{mathitem}
\item[(a)] $C$ contains $P$ as a subpath;
\item[(b)] $C$ contains an adjustable segment $A$ (with end-vertices $x_1,x_2$);
\item[(c)] every edge of $C$ has a vertex in $R$, unless it is an edge in
$A$.
\end{mathitem}

Now we choose a cycle $C$ in $\overline{G}$ satisfying (a)(b)(c)
with order as large as possible but at most $m$. If $\nu(C)=m$, then
we are done. So we assume that $\nu(C)\leq m-1$. We claim that
$V(R)\subset V(C)$. Assume the contrary. Let $v$ be a vertex in
$U=V(R)\backslash V(C)$.

If ($X_1\cap X_2=\emptyset$ and) $A=x_1x'_1ux'_2x_2$ with $x''_1\in
X_1\backslash V(C)$, $x''_2\in X_2\backslash V(C)$, then
$C'=C-x_1x'_1\cup x_1x''_1x'_1$ is a required cycle with order
$\nu(C)+1$, a contradiction. Using the same analysis, we can
conclude that $A=x_1x'_1x''_1ux''_2x'_2x_2$ (if $X_1\cap
X_2=\emptyset$) or $A=x_1x'_1xx'_2x_2$ (if $X_1\cap
X_2\neq\emptyset$).

If $X_1\cap X_2\neq\emptyset$, then $P$ is a longest path of
$\overline{R}$; if $X_1\cap X_2=\emptyset$, then noting that $u,v\in
V(R-P)$, $P$ is a longest path of $\overline{R}$ as well. Thus
$\nu(P)\geq p$ and $U\cup\{s,t\}$ is contained in a component of
$R$. If there is a vertex $y$ in $X$ that is joined to
$U\cup\{s,t\}$, then we use $y$ instead of the vertex $x'_1$, $x'_2$
or $x$ in $C$, for the case $y\in X_1\backslash X_2$, $y\in
X_2\backslash X_1$, or $y\in X_1\cap X_2$, respectively. Thus we
assume that every vertex in $X\backslash\{x'_1,x'_2,x\}$ is not
joined to $U\cup\{s,t\}$.

If every vertex in $X$ is in $V(C)$, then noting that there are at
most 5 vertices in $X$ each of which has a successor on $C$ such
that it is not in $R-P$, we have
$$\nu(C)\geq\nu(P)+|X|+(|X|-5)\geq p+2|X|-5\geq m,$$
a contradiction. So we assume that there is a vertex $x'$ in $X$
which is not in $C$. Let $v'$ be the predecessor of $x_1$ in $C$.
Clearly $v'\in U\cup\{s,t\}$. Then
$C'=v'x'vx_1x''_1\overrightarrow{C}[x''_1,v']$ (if $X_1\cap
X_2=\emptyset$) or $C'=v'x'vx_1x\overrightarrow{C}[x,v']$ (if
$X_1\cap X_2\neq\emptyset$) is a required cycle of order $\nu(C)+1$,
a contradiction. Thus as we claimed, every vertex in $R$ is in $C$.
This implies that $C$ is a cycle in $\overline{G}$ satisfying
\begin{mathitem}
\item[(d)] there is an edge $x_1x'_1\in E(C)$ such that $x_1,x'_1\in X_1$;
\item[(e)] there is an edge $x_2x'_2\in E(C)$ such that $x_2,x'_2\in X_2$;
\item[(f)] $V(R)\subset V(C)$.
\end{mathitem}

Now we choose a cycle $C$ in $\overline{G}$ satisfying (d)(e)(f)
with order as large as possible but at most $m$. If $\nu(C)=m$, then
we are done. So we assume that $\nu(C)\leq m-1$. If every vertex in
$X$ is in $C$, then
$$\nu(C)=\nu(R)+|X|\geq m,$$
a contradiction. So we assume that there is a vertex $x'$ in $X$
which is not in $C$. If $x'\in X_1$, then $C'=C-x_1x'_1\cup
x_1x'x'_1$ is a required cycle of order $\nu(C)+1$; if $x'\in X_2$,
then $C'=C-x_2x'_2\cup x_2x'x'_2$ is a required cycle of order
$\nu(C)+1$, a contradiction.

Thus the lemma holds.
\end{proof}

The proof of the next lemma is similar as the proof of Lemma 8, but
more involved.

\begin{lemma}
Let $G$ be a graph, $R$ be an induced subgraph of $G$, $X_1,X_2$ two
independent sets of $G-R$ (possibly joint), $X=X_1\cup X_2$. If
\begin{mathitem}
\item[(1)] $|X_1|=|X_2|\geq 3$, $|X_1\backslash
X_2|=|X_2\backslash X_1|\geq 2$;
\item[(2)] every component of $R$ has order at least 2;
\item[(3)] every component of $R$ is joined to at most one vertex in $X_i$, $i=1,2$;
\item[(4)] for any component $H$ of $R$, there are at least $q$ vertices
in $G-R$ each of which is either in $X$ or not joined to $H$;
\item[(5)] $\overline{R}$ contains a path of order $p\geq 2$; and
\item[(6)] $m\leq\min\{\lceil3\nu(R)/2\rceil+4,\nu(R)+q-1,p+2q-5\}$,
\end{mathitem}
then $\overline{G}$ contains a $C_m$.
\end{lemma}

\begin{proof}
We use the concept of an adjustable segment defined in Lemma 8. If
$X_1\cap X_2\neq\emptyset$, then let $P$ be a path in $\overline{R}$
with the largest order; if $X_1\cap X_2=\emptyset$, then let $P$ be
a non-Hamilton path in $\overline{R}$ with the largest order.

If $\nu(P)\geq m-5$, then similar as in Lemma 8, we can find a $C_m$
in $\overline{G}$. Thus we assume that $\nu(P)\leq m-6$. By a
similar argument as in Lemma 8, we can get a cycle $C$ in
$\overline{G}$ of order at least $\nu(P)+5$ such that
\begin{mathitem}
\item[(a)] $C$ contains $P$ as a subpath;
\item[(b)] $C$ contains an adjustable segment $A$ (with end-vertices
$x_1,x_2$);
\item[(c)] every edge of $C$ has a vertex in $R$, unless it is an edge in
$A$.
\end{mathitem}

Now we choose a cycle $C$ in $\overline{G}$ satisfying (a)(b)(c)
with order as large as possible but at most $m$. If $\nu(C)=m$, then
we are done. So we assume that $\nu(C)\leq m-1$. We claim that
$V(R)\subset V(C)$. Assume the contrary. Let $v$ be a vertex in
$U=V(R-C)$.

Using the same analysis in Lemma 8, we can conclude that
$A=x_1x'_1x''_1ux''_2x'_2x_2$ (if $X_1\cap X_2=\emptyset$) or
$A=x_1x'_1xx'_2x_2$ (if $X_1\cap X_2\neq\emptyset$) and $P$ is a
longest path of $\overline{R}$. Thus $\nu(P)\geq p$ and
$U\cup\{s,t\}$ is contained in a common component of $R$.
Furthermore, we can assume that every vertex in
$X\backslash\{x'_1,x'_2,x\}$ is not joined to $U\cup\{s,t\}$.

Let $W$ be the union of $X$ and the set of vertices in $G-R$ that
are not joined to $U\cup\{s,t\}$. Then $|W|\geq q$. If every vertex
in $W$ is in $V(C)$, then noting that there are at most 5 vertices
in $W$ each of which has a successor on $C$ such that it is not in
$R-P$, we have
$$\nu(C)\geq\nu(P)+|W|+(|W|-5)\geq p+2q-5\geq m,$$
a contradiction. So we assume that there is a vertex $w$ in $W$ that
is not in $V(C)$. Let $v'$ be the predecessor of $x_1$ in $C$.
Clearly $v'\in U\cup\{s,t\}$. Then
$C'=v'wvx_1x''_1\overrightarrow{C}[x''_1,v']$ (if $X_1\cap
X_2=\emptyset$) or $C'=v'wvx_1x\overrightarrow{C}[x,v']$ (if
$X_1\cap X_2\neq\emptyset$) is a required cycle of order $\nu(C)+1$,
a contradiction. Thus as we claimed, every vertex in $R$ is in $C$.
This implies $C$ satisfies (b)(c) and
\begin{mathitem}
\item[(d)] $V(R)\subset V(C)$.
\end{mathitem}
Now we choose a cycle $C$ in $\overline{G}$ satisfying (b)(c)(d)
with order as large as possible but at most $m$. If $\nu(C)=m$, then
we are done. So we assume that $\nu(C)\leq m-1$. By a similar
argument as above, we can conclude that
$A=x_1x'_1x''_1ux''_2x'_2x_2$ (if $X_1\cap X_2=\emptyset$) or
$A=x_1x'_1xx'_2x_2$ (if $X_1\cap X_2\neq\emptyset$).

We claim that there are two vertices $u_1,u_2$ in $C$ such that
$u_1,u_2$ are in a common component of $R$ and $u_1^+,u_2^+\in
V(R)$. Assume the contrary. Note that every component of $R$ has at
least 2 vertices, there is at most one vertex in a component, such
that it has a successor on $C$ in $R$, and there are 4 vertices of
$C$ (in the adjusted segment) each of which is not a successor of
some vertex in $R$. Thus
$$\nu(C)\geq\nu(R)+\left\lceil\frac{\nu(R)}{2}\right\rceil+4=
    \left\lceil\frac{3\nu(R)}{2}\right\rceil+4\geq m,$$
a contradiction. Thus as we claimed, there are two edges
$u_1u_1^+,u_2u_2^+$ such that $u_1,u_2$ are in a common component of
$R$ and $u_1^+,u_2^+\in V(R)$.

If there is a vertex $y$ in $X\backslash V(C)$ that is joined to
$\{u_1,u_2\}$, then we use $y$ instead of the vertex $x'_1$, $x'_2$
or $x$ in $C$. Thus we assume that every vertex in $X\backslash
V(C)$ is not joined to $\{u_1,u_2\}$. Let $W$ be the union of $X$
and the set of vertices in $G-R$ that are not joined to
$\{u_1,u_2\}$. Then $|W|\geq q$. If every vertex in $W$ is in $C$,
then
$$\nu(C)\geq\nu(R)+|W|\geq \nu(R)+q\geq m,$$
a contradiction. Thus we assume that there is a vertex $w$ in $W$
that is not in $C$.

If $u_1^+,u_2^+$ are in distinct components of $R$, then
$C'=u_1wu_2\overleftarrow{C}[u_2,u_1^+]u_1^+u_2^+\overrightarrow{C}[u_2^+,u_1]$
is a required cycle with order $\nu(C)+1$. Now we assume that
$u_1^+,u_2^+$ are in a common component of $R$.

If there is a vertex $y'$ in $X\backslash\{w\}$ that is joined to
$\{u_1^+,u_2^+\}$, then we use $y'$ instead of the vertex $x'_1$,
$x'_2$ or $x$ in $C$. Thus we assume that every vertex in
$X\backslash V(C)\backslash\{w\}$ is not joined to
$\{u_1^+,u_2^+\}$.

Let $W'$ be the union of $X$ and the set of vertices in $G-R$ that
are not joined to $\{u_1^+,u_2^+\}$. Then $|W'|\geq q$. If every
vertex in $W'\backslash\{w\}$ is in $C$, then
$$\nu(C)\geq\nu(R)+|W'|-1\geq\nu(R)+q-1\geq m,$$
a contradiction. Thus we assume that there is a vertex $w'$ in
$W\backslash\{w\}$ that is not in $C$. Let
$C'=u_1wu_2\overleftarrow{C}[u_2,u_1^+]u_1^+w'u_2^+\overrightarrow{C}[u_2^+,u_1]$.
Then $C''=C'-x_1x'_1x''_1\cup x_1x''_1$ (if $X_1\cap X_2=\emptyset$)
or $C''=C'-x_1x'_1x\cup x_1x$ (if $X_1\cap X_2\neq\emptyset$) is a
required cycle of order $\nu(C)+1$, a contradiction.
\end{proof}

\section{Proof of Theorem 5}

The case of $n=2$ is trivial. For the case of $n=3$ or $n=4$, we are
done by Theorem 6. Thus in the following we will assume that $n\geq
5$.

Let $t=t(n,m)$. By Lemma 1, $t(n,m)=\min\{t: t\notin \mathcal
{L}[t-m+1,n-1]\}$. Thus $t-1\in \mathcal {L}[t-m,n-1]$. Let
$t-1=\sum_{i=1}^k t_i$, where $t_i\in[t-m,n-1]$, $1\leq i\leq k$.
Let $G$ be a graph with $k$ components $H_1,\ldots,H_k$ such that
$H_i$ is a clique on $t_i$ vertices. Note that $G$ contains no $P_n$
since every component of $G$ has less than $n$ vertices; and
$\overline{G}$ contains no $W_m$ since every vertex of $G$ has less
than $m$ nonadjacent vertices. Thus $G$ is a graph on $t-1$ vertices
such that $G$ contains no $P_n$ and $\overline{G}$ contains no
$W_m$. This implies that $R(P_n,W_m)\geq t$.

Now we will prove that $R(P_n,W_m)\leq t$. Assume not. Let $G$ be a
graph on $t$ vertices such that $G$ contains no $P_n$ and
$\overline{G}$ contains no $W_m$.

Let $s=m+n-t$ (i.e., $\nu(G)=m+n-s$).

\begin{claim}
  $1\leq s\leq \lfloor(n+5)/4\rfloor$.
\end{claim}

\begin{proof}
Let $t'=m+n-1$. Since $t'-m+1=n$, $[t'-m+1,n-1]=\emptyset$, and
$t'\notin \mathcal {L}(\emptyset)=\{0\}$, we have $t\leq t'=m+n-1$,
and this implies that $s\geq 1$.

Now we prove that $s\leq (n+5)/4$. By Lemma 1, $t\notin \mathcal
{L}[t-m+1,n-1]$. Thus $t\notin[k(t-m+1),k(n-1)]$, for every $k\geq
1$. That is, $t\in[k(n-1)+1,(k+1)(t-m+1)-1]$, for some $k$.

If $k\leq 2$, then by $t\leq (k+1)(t-m+1)-1$, we get that
$$t\geq\frac{k+1}{k}m-1\geq\frac{3}{2}m-1>3n-1\geq 3(t-m+1)-1,$$
a contradiction. Thus we assume that $k\geq 3$.

If $m\leq(k^2n-k^2+2k)/(k+1)$, then
\begin{align*}
  s &=m+n-t\leq \frac{k^2n-k^2+2k}{k+1}+n-(k(n-1)+1)\\
    &=\frac{n+2k-1}{k+1}\leq\frac{n+5}{4}.
\end{align*}

If $m>(k^2n-k^2+2k)/(k+1)$, then
\begin{align*}
  s &=m+n-t\leq m+n-(\frac{k+1}{k}m-1)\\
    &=n-\frac{m}{k}+1<n-\frac{k^2n-k^2+2k}{k(k+1)}+1\\
    &=\frac{n+2k-1}{k+1}\leq\frac{n+5}{4}.
\end{align*}

Thus the claim holds.
\end{proof}

We list the possible values of $s$ for $n\leq 16$.

\begin{center}
\begin{tabular}{c|c|c|c|c|c|c|c|c|c|c|c|c}
  \hline \hline
  $n$ & 5 & 6 & 7 & 8 & 9 & 10 & 11 & 12 & 13 & 14 & 15 & 16\\
  \hline
  $s\leq$ & 2 & 2 & 3 & 3 & 3 & 3 & 4 & 4 & 4 & 4 & 5 & 5\\
  \hline \hline
\end{tabular}\\[3mm]
\footnotesize {Table 1: The possible values of $s$ for $n\leq 16$.}
\end{center}

\begin{claim}
Let $v$ be an arbitrary vertex of $G$ and $G'\subset G-v-N(v)$. Then
$\overline{G'}$ contains no $C_m$.
\end{claim}

\begin{proof}
  Otherwise, noting that $v$ is nonadjacent to every vertex in the
  $C_m$, there will be a $W_m$ in $\overline{G}$ (with the hub $v$).
\end{proof}

\begin{claim}
  $\delta(G)\geq \lceil n/2\rceil-s+1$.
\end{claim}

\begin{proof}
Assume the contrary. Let $v$ be a vertex of $G$ with $d(v)\leq
\lceil n/2\rceil-s$. Then $G'=G-v-N(v)$ has at least $m+\lfloor n/2
\rfloor-1$ vertices. Since $G'$ contains no $P_n$, by Theorem 2,
$\overline{G'}$ contains a $C_m$ (note that $m\geq 2n+1$), a
contradiction to Claim 2.
\end{proof}

From Claims 1 and 3, one can see that $\delta(G)\geq 2$ (when $n\geq
5$).

\begin{case}
  $G$ is disconnected.
\end{case}

\begin{subcase}
  Every component of $G$ has order less than $n$.
\end{subcase}

Let $H_i$, $1\leq i\leq k$, be the components of $G$. Since $t\notin
\mathcal {L}[t-m+1,n-1]$, there is a component, say $H_1$, with
order at most $t-m$. Thus $\sum_{i=2}^k\nu(H_i)\geq m$. Since
$\nu(H_i)\leq n-1\leq\lfloor m/2\rfloor$. By Lemma 3,
$\overline{G-H_1}$ contains a $C_m$, a contradiction.

\begin{subcase}
  There is a component of $G$ with order at least $n$.
\end{subcase}

Let $H$ be a component of $G$ with the largest order. Note that
$\nu(H)\geq n$. If every vertex of $H$ has degree at least $\lfloor
n/2\rfloor$, then by Lemma 2, $H$ contains a $P_n$, a contradiction.
Thus there is a vertex $v$ in $H$ with $d(v)\leq\lfloor
n/2\rfloor-1$. Let $G'=G-v-N(v)$. Then
\begin{align*}
\nu(G') & =\nu(G)-1-d(v)\\
        & \geq m+n-s-1-\left\lfloor\frac{n}{2}\right\rfloor+1\\
        & =m+\left\lceil\frac{n}{2}\right\rceil-s\geq m.
\end{align*}

Since $\nu(H)\geq n>1+d(v)$, $G'$ is disconnected. Let $\mathcal{H}$
be the union of $\omega(G')-1$ components of $G'$. We will prove
that $\nu(\mathcal{H})\geq m+\lfloor n/2\rfloor-\nu(G')$.

Let $H'$ be a component of $G$ other than $H$. If
$H'\subset\mathcal{H}$, then $\nu(\mathcal{H})\geq\nu(H')\geq
1+\delta(G)$, and
\begin{align*}
        & \nu(\mathcal{H})+\nu(G')-m-\left\lfloor\frac{n}{2}\right\rfloor
            \geq 1+\delta(G)+\nu(G')-m-\left\lfloor\frac{n}{2}\right\rfloor\\
\geq    & 1+\left\lceil\frac{n}{2}\right\rceil-s+1+
            m+\left\lceil\frac{n}{2}\right\rceil
            -s-m-\left\lfloor\frac{n}{2}\right\rfloor\\
=       & \left\lceil\frac{n}{2}\right\rceil+\pa(n)+2-2s\geq 0.
\end{align*}
If $H'\not\subset\mathcal{H}$, then
$\nu(\mathcal{H})=\nu(G')-\nu(H')\geq\nu(G')-\lfloor\nu(G)/2\rfloor$,
and
\begin{align*}
        & \nu(\mathcal{H})+\nu(G')-m-\left\lfloor\frac{n}{2}\right\rfloor
            \geq \nu(G')-\left\lfloor\frac{\nu(G)}{2}\right\rfloor+\nu(G')
            -m-\left\lfloor\frac{n}{2}\right\rfloor\\
\geq    & 2\left(m+\left\lceil\frac{n}{2}\right\rceil-s\right)-
            \left\lfloor\frac{m+n-s}{2}\right\rfloor-m
            -\left\lfloor\frac{n}{2}\right\rfloor\\
=       & \left\lceil\frac{n}{2}\right\rceil+\pa(n)+
            \left\lceil\frac{m-n-3s}{2}\right\rceil\\
\geq    & \left\lceil\frac{m-3s}{2}\right\rceil
            \geq\left\lceil\frac{2n+1-3s}{2}\right\rceil\geq 0.
\end{align*}

Now by Lemma 6, $\overline{G'}$ contains a $C_m$, a contradiction.

\begin{case}
  $G$ has connectivity 1.
\end{case}

Note that $\delta(G)\geq\lceil n/2\rceil-s+1\geq 2$. Every end-block
of $G$ is 2-connected.

\begin{subcase}
  There is an end-block of $G$ with order at least $\lceil m/2\rceil+1$.
\end{subcase}

Let $B$ be an end-block of $G$ with the maximum order, and $x$ be
the cut-vertex of $G$ contained in $B$. Let $x'$ be a cut-vertex of
$G$ such that the longest path between $x$ and $x'$ is as long as
possible. Clearly $x'$ is contained in some end-blocks. Let $B'$ be
an end-block of $G$ containing $x'$ ($B\neq B'$). Let $v$ be a
vertex in $B-x$ such that $d_{B-x}(v)$ is as small as possible.

\begin{claim}
$$d_{B-x}(v)\leq\left\{\begin{array}{ll}
\lceil(n+2s-\pa(n))/4\rceil-2,      & \mbox{if } x=x';\\
\lfloor(n+2s-\pa(n))/4\rfloor-2,    & \mbox{if } xx' \mbox{ is a cut-edge of } G;\\
\lceil(n+2s-\pa(n))/4\rceil-3,      & \mbox{otherwise}.
\end{array}\right.$$
\end{claim}

\begin{proof}
We set a parameter $a$ such that $a=0$ if $x=x'$, 1 if $xx'$ is a
cut-edge of $G$, and 2 otherwise. So there is a path between $x$ and
$x'$ of length at least $a$.

By Claim 3 and Lemma 2, $B'$ contains a path from $x'$ of order at
least $\lceil n/2\rceil-s+2$, and $G-(B-x)$ contains a path from $x$
of order at least $\lceil n/2\rceil-s+a+2$.

Note that $\nu(B)\geq\lceil m/2\rceil+1\geq\lfloor
n/2\rfloor+s-a-1$. If $\delta(B-x)\geq\lfloor(\lfloor
n/2\rfloor+s-a-1)/2\rfloor$, then by Lemma 2, $B$ contains a path
from $x$ of order at least $\lfloor n/2\rfloor+s-a-1$. Thus $G$
contains a $P_n$, a contradiction. This implies that
$$\delta(B-x)\leq\left\lfloor\frac{\lfloor
n/2\rfloor+s-a-1}{2}\right\rfloor-1
=\left\lceil\frac{n+2s-\pa(n)-2a}{4}\right\rceil-2.
$$
Thus the claim holds.
\end{proof}

Note that
\begin{align*}
        & \nu(B-x-v-N(v))=\nu(B)-2-d_{B-x}(v)\\
\geq    & \left\lceil\frac{m}{2}\right\rceil+1-2
            -\left\lceil\frac{n+2s-\pa(n)}{4}\right\rceil+2\\
\geq    & \left\lceil\frac{m}{2}-\frac{n+2s-\pa(n)+2}{4}\right\rceil
            +1\geq 1.
\end{align*}
This implies that $V(B)\backslash\{x,v\}\backslash
N(v)\neq\emptyset$.

\begin{subsubcase}
  $x=x'$.
\end{subsubcase}

In this case, $G$ has only one cut-vertex $x$. Let $G'=G-x-v-N(v)$.
Then $G'$ is disconnected and
\begin{align*}
\nu(G') & =\nu(G)-2-d_{B-x}(v)\\
        & \geq m+n-s-2-\left\lceil\frac{n+2s-\pa(n)}{4}\right\rceil+2\\
        & =m+\left\lfloor\frac{3n+\pa(n)-6s}{4}\right\rfloor\geq m.
\end{align*}

Let $\mathcal {H}$ be the union of any $\omega(G')-1$ components of
$G'$. We will prove that $\nu(\mathcal{H})\geq m+\lfloor
n/2\rfloor-\nu(G')$.

If $B'-x\not\subset\mathcal{H}$, then
$\nu(\mathcal{H})=\nu(G')-\nu(B'-x)\geq\nu(G')-
\lfloor(\nu(G)-1)/2\rfloor$, and
\begin{align*}
        & \nu(\mathcal{H})+\nu(G')-m-\left\lfloor\frac{n}{2}\right\rfloor
            \geq\nu(G')-\left\lfloor\frac{\nu(G)-1}{2}\right\rfloor+\nu(G')
            -m-\left\lfloor\frac{n}{2}\right\rfloor\\
\geq    &
            2\left(m+\left\lfloor\frac{3n+\pa(n)-6s}{4}\right\rfloor\right)
            -\left\lfloor\frac{m+n-s-1}{2}\right\rfloor-m
            -\left\lfloor\frac{n}{2}\right\rfloor\\
\geq    & \left\lceil m+2\cdot\frac{3n-6s-2}{4}
            -\frac{m+n-s-1}{2}-\frac{n}{2}\right\rceil\\
=       & \left\lceil\frac{m+n-5s-1}{2}\right\rceil\geq
            \left\lceil\frac{3n-5s}{2}\right\rceil\geq 0.
\end{align*}

Now we assume that $B'-x\subset\mathcal{H}$. In this case
$\nu(\mathcal{H})\geq\nu(B'-x)\geq\delta(G)$, and
\begin{align*}
        & \nu(\mathcal{H})+\nu(G')-m-\left\lfloor\frac{n}{2}\right\rfloor
            \geq \delta(G)+\nu(G')-m-\left\lfloor\frac{n}{2}\right\rfloor\\
\geq    &
\left\lceil\frac{n}{2}\right\rceil-s+1+m+\left\lfloor\frac{3n+\pa(n)-6s}{4}\right\rfloor-
            m-\left\lfloor \frac{n}{2}\right\rfloor\\
\geq    & \left\lfloor\frac{3n+5\pa(n)+4-10s}{4}\right\rfloor.
\end{align*}
Note that $3n+5\pa(n)+4-10s\geq 0$ unless $n=8$ and $s=3$.

\noindent\textbf{Petty Case.} $n=8$ and $s=3$.

In this case $\nu(B'-x)\geq 2$ and $d_{B-x}(v)\leq 2$. If
$\nu(\mathcal{H})\geq 3$, or if $d_{B-x}(v)=1$, then it is easy to
see that $\nu(\mathcal{H})\geq m+\lfloor n/2\rfloor-\nu(G')$. Now we
assume that $\nu(B'-x)=\nu(\mathcal{H})=2$ and $d_{B-x}(v)=2$. This
implies that $B'$ is a triangle, there are only two blocks $B,B'$,
and every vertex in $B-x$ has degree at least 2 in $B-x$. If $B-x$
has a cut-vertex, then noting that every end-block of $B-x$ has at
least three vertices, $B$ contains a path from $x$ of order at least
6, and $G$ contains a $P_8$, a contradiction. So we assume that
$B-x$ is 2-connected.

Note that $B-x$ contains a cycle of order at least 4. Let $C$ be a
longest cycle of $B-x$. If $\nu(C)\geq 5$, then there is also an
path from $x$ in $B$ of order at least 6, a contradiction. Thus we
assume that $\nu(C)=4$. If there is a component of $B-x-C$ with
order at least 2, or if there is a vertex in $B-x-C$ adjacent to two
consecutive vertices on $C$, then it is easy to find a cycle longer
than $C$. Thus $B-x-C$ consists of isolated vertices and every
vertex is adjacent to two nonconsecutive vertices on $C$. If there
are two vertices in $B-x-C$ adjacent to different vertices on $C$,
we can also find a longer cycle. Thus all the vertices of $B-x-C$
have the same neighbors on $C$. This implies that $B-x-v-N(v)$ is
disconnected and then $\nu(\mathcal{H})\geq\nu(B'-x)+1=3$. Thus we
also have $\nu(\mathcal{H})\geq m+\lfloor n/2\rfloor-\nu(G')$.

By Lemma 6, $\overline{G'}$ contains a $C_m$, a contradiction.

\begin{subsubcase}
  $xx'$ is a cut-edge of $G$ and there is only one end-block containing $x'$.
\end{subsubcase}

By Claim 4, $d_{B-x}(v)\leq\lfloor(n+2s-\pa(n))/4\rfloor-2$. Let
$G'=G-x-v-N(v)$. Then
\begin{align*}
\nu(G') &=\nu(G)-2-d_{B-x}(v)\\
        & \geq m+n-s-2-\left\lfloor\frac{n+2s-\pa(n)}{4}\right\rfloor+2\\
        & =m+\left\lceil\frac{3n+\pa(n)-6s}{4}\right\rceil\geq m.
\end{align*}

Now let $\mathcal {H}$ be the union of any $\omega(G')-1$ components
of $G'$. If $B'\subset\mathcal {H}$, then
$\nu(\mathcal{H})\geq\nu(B')\geq\delta(G)+1$, and
\begin{align*}
        & \nu(\mathcal{H})+\nu(G')-m-\left\lfloor\frac{n}{2}\right\rfloor
            \geq\delta(G)+1+\nu(G')-m-\left\lfloor\frac{n}{2}\right\rfloor\\
\geq    & \left\lceil\frac{n}{2}\right\rceil-s+2+
            m+\left\lceil\frac{3n+\pa(n)-6s}{4}\right\rceil-
            m-\left\lfloor\frac{n}{2}\right\rfloor\\
=       & \left\lceil\frac{3n+5\pa(n)+8-10s}{4}\right\rceil\geq 0.
\end{align*}
If $B'\not\subset\mathcal{H}$, then
$\nu(\mathcal{H})=\nu(G')-\nu(B')\geq\nu(G')-\lfloor\nu(G)/2\rfloor$,
and
\begin{align*}
        &   \nu(\mathcal{H})+\nu(G')-\left\lfloor\frac{n}{2}\right\rfloor-m
            \geq\nu(G')-\left\lfloor\frac{\nu(G)}{2}\right\rfloor+\nu(G')
            -m-\left\lfloor\frac{n}{2}\right\rfloor\\
\geq    &
2\left(m+\left\lceil\frac{3n+\pa(n)-6s}{4}\right\rceil\right)-
            \left\lfloor\frac{m+n-s}{2}\right\rfloor
            -m-\left\lfloor\frac{n}{2}\right\rfloor\\
\geq    & \left\lceil m+2\cdot\frac{3n+\pa(n)-6s}{4}
            -\frac{m+n-s}{2}-\frac{n}{2}\right\rceil\\
=       & \left\lceil\frac{m+n+\pa(n)-5s}{2}\right\rceil\geq
            \left\lceil\frac{3n+\pa(n)+1-5s}{2}\right\rceil\geq 0.
\end{align*}

By Lemma 6, $\overline{G'}$ contains a $C_m$, a contradiction.

\begin{subsubcase}
  $xx'\notin E(G)$, or $xx'$ is not a cut-edge of $G$, or there
  are at least two end-blocks containing $x'$.
\end{subsubcase}

Let $G'=G-x-x'-v-N(v)$. Note that in this case $\omega(G')\geq 3$,
and we have
\begin{align*}
    \nu(G') &=\nu(G)-3-d_{B-x}(v)\\
    & \geq m+n-s-3-\left\lfloor\frac{n+2s-\pa(n)}{4}\right\rfloor+2\\
    & =m+\left\lceil\frac{3n+\pa(n)-6s-4}{4}\right\rceil\geq m.
\end{align*}

Now let $\mathcal {H}$ be the union of any $\omega(G')-1$ components
of $G'$. If $B'-x'\subset\mathcal {H}$, then noting that
$\omega(G')\geq 3$,
$\nu(\mathcal{H})\geq\nu(B'-x')+1\geq\delta(G)+1$, and
\begin{align*}
        & \nu(\mathcal{H})+\nu(G')-m-\left\lfloor\frac{n}{2}\right\rfloor
            \geq\delta(G)+1+\nu(G')-m-\left\lfloor\frac{n}{2}\right\rfloor\\
\geq    & \left\lceil\frac{n}{2}\right\rceil-s+2+
            m+\left\lceil\frac{3n+\pa(n)-6s-4}{4}\right\rceil-
            m-\left\lfloor\frac{n}{2}\right\rfloor\\
=       & \left\lceil\frac{3n+5\pa(n)+4-10s}{4}\right\rceil\geq 0.
\end{align*}
If $B'-x'\not\subset\mathcal{H}$, then
$\nu(\mathcal{H})=\nu(G')-\nu(B'-x')\geq\nu(G')-\lfloor\nu(G)/2\rfloor+1$,
and
\begin{align*}
        & \nu(\mathcal{H})+\nu(G')-m-\left\lfloor\frac{n}{2}\right\rfloor
            \geq\nu(G')-\left\lfloor\frac{\nu(G)}{2}\right\rfloor+1+\nu(G')
            -m-\left\lfloor\frac{n}{2}\right\rfloor\\
\geq    &
            2\left(m+\left\lceil\frac{3n+\pa(n)-6s-4}{4}\right\rceil\right)
            -\left\lfloor\frac{m+n-s}{2}\right\rfloor+1
            -m-\left\lfloor\frac{n}{2}\right\rfloor\\
\geq    & \left\lceil m+2\cdot\frac{3n-6s-4}{4}-\frac{m+n-s}{2}+1
            -\frac{n}{2}\right\rceil\\
=       & \left\lceil\frac{m+n-5s-2}{2}\right\rceil\geq
            \left\lceil\frac{3n-5s-1}{2}\right\rceil\geq 0.
\end{align*}

By Lemma 6, $\overline{G'}$ contains a $C_m$, a contradiction.

\begin{subcase}
  Every end-block of $G$ has order at most $\lceil m/2\rceil$.
\end{subcase}

\begin{claim}
Let $G'$ be a disconnected subgraph of $G$. If
\begin{mathitem}
\item[(a)] $\nu(G')\geq m$; and
\item[(b)] there are two components of $G'$, each of which is an end-block
removing a cut-vertex of $G$ contained in the end-block,
\end{mathitem}
then the order sum of every $\omega(G')-1$ components in $G'$ is
$m+\lfloor n/2\rfloor-\nu(G')$.
\end{claim}

\begin{proof}
Let $B-x$ and $B'-x'$ be two components of $G'$, where $B,B'$ are
two end-blocks of $G$ and $x,x'$ are two cut-vertices of $G$
contained in $B$ and $B'$, respectively.

Let $\mathcal {H}$ be the union of any $\omega(G')-1$ components of
$G'$. We first assume that both $B-x$ and $B'-x'\subset\mathcal
{H}$. Then $\nu(\mathcal{H})\geq\nu(B-x)+\nu(B'-x')\geq 2\delta(G)$,
and
\begin{align*}
        & \nu(\mathcal{H})+\nu(G')-m-\left\lfloor\frac{n}{2}\right\rfloor
            \geq 2\delta(G)+\nu(G')-m-\left\lfloor\frac{n}{2}\right\rfloor\\
\geq    & 2\left(\left\lceil\frac{n}{2}\right\rceil-s+1\right)+
            m-m-\left\lfloor\frac{n}{2}\right\rfloor\\
=       & \left\lceil\frac{n}{2}\right\rceil+\pa(n)+2-2s\geq 0.
\end{align*}
Now we assume that $\mathcal{H}$ does not contain $B-x$ or $B'-x'$.
Without loss of generality, we assume that $\mathcal{H}$ does not
contain $B-x$. Then
$\nu(\mathcal{H})=\nu(G')-\nu(B-x)\geq\nu(G')-\lceil m/2\rceil+1$,
and
\begin{align*}
        & \nu(\mathcal{H})+\nu(G')-m-\left\lfloor\frac{n}{2}\right\rfloor
            \geq\nu(G')-\left\lceil\frac{m}{2}\right\rceil+1+\nu(G')-
            m-\left\lceil\frac{n}{2}\right\rceil\\
\geq    & 2m-\left\lceil\frac{m}{2}\right\rceil+1-
            m-\left\lfloor\frac{n}{2}\right\rfloor
            =\left\lfloor\frac{m}{2}\right\rfloor-
            \left\lfloor\frac{n}{2}\right\rfloor+1\geq 0.
\end{align*}

Thus the claim holds.
\end{proof}

\begin{subsubcase}
  $G$ has only two end-blocks.
\end{subsubcase}

Let $B$ and $B'$ be the two end-blocks of $G$, and let $x$ and $x'$
be the cut-vertices of $G$ contained in $B$ and $B'$, respectively.
Note that
$$\nu(G)-\nu(B)-\nu(B')\geq
m+n-s-2\cdot\left\lceil\frac{m}{2}\right\rceil=n-s-\pa(m)\geq 1.$$
This implies that $V(G)\backslash V(B)\backslash
V(B')\neq\emptyset$.

Note that in this case $G-(B-x)-(B'-x')+xx'$ is 2-connected. If
every vertex in $G-B-B'$ has degree at least $2s-\pa(n)-3$, then by
Lemma 2, there is a path from $x$ to $x'$ of order at least
$2s-\pa(n)-2$. Note that $B$ contains a path from $x$ of order at
least $\lceil n/2\rceil-s+2$, and $B'$ contains a path from $x'$ of
order at least $\lceil n/2\rceil-s+2$. Thus $G$ contains a $P_n$, a
contradiction. This implies that there is a vertex $v$ in $G-B-B'$
with $d(v)\leq 2s-\pa(n)-4$.

Let $G'=G-x-x'-v-N(v)$. Then
\begin{align*}
\nu(G') & \geq\nu(G)-3-d(v)\\
        & \geq m+n-s-3-2s+\pa(n)+4\\
        & =m+n+\pa(n)+1-3s\geq m.
\end{align*}

By Claim 5, the order sum of every $\omega(G')-1$ components in $G'$
is at least $m+\lfloor n/2\rfloor-\nu(G')$. By Lemma 6,
$\overline{G'}$ contains a $C_m$, a contradiction.

\begin{subsubcase}
  $G$ has at least three end-blocks.
\end{subsubcase}

Let $x$ and $x'$ be two cut-vertices of $G$ such that the longest
path between $x$ and $x'$ in $G$ is as long as possible. Clearly $x$
and $x'$ are both contained in some end-blocks. Let $B$ and $B'$ be
two end-blocks of $G$ containing $x$ and $x'$, respectively. Let $v$
be a vertex in $V(B-x)\cup V(B'-x')$ such that $d_{G-x-x'}(v)$ is as
small as possible. We assume without loss of generality that $v\in
V(B-x)$.

\begin{claim}
$$d_{B-x}(v)\leq\left\{\begin{array}{ll}
\lfloor n/2\rfloor-2,   & \mbox{if } x=x';\\
\lceil n/2\rceil-3,     & \mbox{if } xx' \mbox{ is a cut-edge of } G;\\
\lfloor n/2\rfloor-3,   & \mbox{otherwise}.
\end{array}\right.$$
\end{claim}

\begin{proof}
We set a parameter $a$ such that $a=0$ if $x=x'$, 1 if $xx'$ is a
cut-edge of $G$, and 2 otherwise. So there is a path between $x$ and
$x'$ of length at least $a$.

If $\delta(B-x)\geq\lfloor(n-a)/2\rfloor-1$, then
$\delta(B'-x')\geq\lfloor(n-a)/2\rfloor-1$. By Lemma 2, $B$ contains
a path from $x$ of order at least $\lfloor(n-a)/2\rfloor+1$ and $B'$
contains a path from $x'$ of order at least
$\lfloor(n-a)/2\rfloor+1$. Thus $G$ contains a path of order at
least $n+1-\pa(n-a)\geq n$, a contradiction. Now we obtain that
$\delta(B-x)\leq\lfloor(n-a)/2\rfloor-2$.
\end{proof}

\begin{subsubsubcase}
  $x=x'$.
\end{subsubsubcase}

Let $G'=G-x-v-N(v)$. Then
\begin{align*}
\nu(G') & =\nu(G)-2-d_{B-x}(v)\\
        & \geq m+n-s-2-\left\lfloor\frac{n}{2}\right\rfloor+2\\
        & =m+\left\lceil\frac{n}{2}\right\rceil-s\geq m.
\end{align*}

Note that every end-block of $G$ other than $B$ removing $x$ is a
component of $G'$. By Claim 5 and Lemma 6, $\overline{G'}$ contains
a $C_m$, a contradiction.

\begin{subsubsubcase}
  $xx'$ is a cut-edge of $G$.
\end{subsubsubcase}

In this case, $G$ has the only two cut-vertices $x$ and $x'$. Let
$G'=G-x-x'-v-N(v)$. Then
\begin{align*}
\nu(G') & =\nu(G)-3-d_{B-x}(v)\\
        & \geq m+n-s-3-\left\lceil\frac{n}{2}\right\rceil+3\\
        & =m+\left\lfloor\frac{n}{2}\right\rfloor-s\geq m.
\end{align*}

Note that every end-block of $G$ other than $B$ removing $x$ or $x'$
is a component of $G'$. By Claim 5 and Lemma 6, $\overline{G'}$
contains a $C_m$, a contradiction.

\begin{subsubsubcase}
  $xy\notin E(G)$ or $xy$ is not a cut-edge of $G$.
\end{subsubsubcase}

Let $B''$ be an end-block of $G$ other than $B$ and $B'$, and let
$x''$ be the cut-vertex of $G$ contained in $B''$ (possibly $x''=x$
or $x'$). Let $G'=G-x-x'-x''-v-N(v)$. Then
\begin{align*}
\nu(G') & \geq\nu(G)-4-d_{B-x}(v)\\
        & \geq m+n-s-4-\left\lfloor\frac{n}{2}\right\rfloor+3\\
        & =m+\left\lceil\frac{n}{2}\right\rceil-s-1\geq m.
\end{align*}

Note that $B'-x'$ and $B''-x''$ are two components of $G'$. By Claim
6 and Lemma 6, $\overline{G'}$ contains a $C_m$, a contradiction.

\begin{case}
  $G$ is 2-connected.
\end{case}

By Claim 3 and Lemma 2, $G$ contains a cycle of order at least
$2(\lceil n/2\rceil-s+1)=n-2s+\pa(n)+2$. Let $C$ be a longest cycle
of $G$ (with a given orientation). Suppose that $\nu(C)=n-r$, where
$$r\leq 2s-\pa(n)-2.$$

Let $H$ be a subgraph of a component of $G-C$, and let
$N_C(H)=\{z_1,z_2,\ldots z_k\}$, where $k=d_C(H)$, and $z_i$, $1\leq
i\leq k$, are in order along $C$. We call the subpath
$\overrightarrow{C}[z_i,z_{i+1}]$ (the indices are taken modulo $k$)
a \emph{good segment} of $C$ (with respect to $H$); moreover, if
$z_i$ and $z_{i+1}$ are joined to two distinct vertices $x,y$ in
$H$, then we call $\overrightarrow{C}[z_i,z_{i+1}]$ a \emph{better
segment} of $C$ (with respect to $H$); moreover, if there is a path
from $x$ to $y$ in $G-C$ of order at least 3, then we call
$\overrightarrow{C}[z_i,z_{i+1}]$ a \emph{best segment} of $C$ (with
respect to $H$). Since $G$ is 2-connected, we conclude that for any
component $H$ of $G-C$, there are at least two good (better, best)
segments of $C$ with respect to $H$ if $\nu(H)\geq 1$ ($\nu(H)\geq
2$, $\nu(H)\geq 3$ and $H$ is not a star, respectively). Note that
every good (better, best) segment has order at least 3 (4, 5,
respectively).

For a vertex $x$ of $C$, we use $x^+$ to denote the successor, and
$x^-$ the predecessor, of $x$ on $C$. For a subset $X$ of $V(C)$, we
set $X^+=\{x^+: x\in X\}$ and $X^-=\{x^-: x\in X\}$.

Now we consider a component $H$ of $G-C$. If $H$ is non-separable,
then $H$ is a $K_1$, a $K_2$ or 2-connected; if $H$ is separable,
then $H$ has at least two end-blocks. In the later case, we call an
end-block of $H$ removing the cut-vertex contained in the end-block
a \emph{branch} of $H$ (also, of $G-C$).

\begin{claim}
Let $H$ be a component of $G-C$ and $u\in V(H)$.
\begin{mathitem}
\item[(1)] If $H$ is non-separable, then $H$ contains a path from
$u$ of order at least $\min\{\nu(H),\lceil r/2\rceil\}$.
\item[(2)] If $H$ is separable and $D$ is a branch of $H$ not
containing $u$, then $H$ contains a path from $u$ of order at least
$\min\{\nu(D)+1,\lceil r/2\rceil\}$.
\end{mathitem}
\end{claim}

\begin{proof}
We first claim that for any two vertices $u,v\in V(H)$,
$d_H(u)+d_H(v)\geq \lceil r/2\rceil$, unless $uv$ is a cut-edge of
$H$. Assume that $uv$ is not a cut-edge of $H$. Then $H$ contains a
path from $u$ to $v$ of order at least 3. Let
$N_C(\{u,v\})=\{z_1,z_2,\ldots,z_k\}$, where $z_i$, $1\leq i\leq k$,
are in order along $C$. If $z_i$ is joined to exactly one vertex of
$u,v$, then $\overrightarrow{C}[z_i,z_{i+1}]$ is a good segment of
$C$ with respect to $\{u,v\}$; if $z_i$ is adjacent to both $u$ and
$v$, then $\overrightarrow{C}[z_i,z_{i+1}]$ is a best segment with
respect to $\{u,v\}$. This implies that
$d_C(u)+d_C(v)\leq\lfloor(n-r)/2\rfloor$ and
\begin{align*}
        & d_H(u)+d_H(v)=d(u)+d(v)-d_C(u)-d_C(v)\\
\geq    & 2\cdot\left(\left\lceil\frac{n}{2}\right\rceil-s+1\right)
            -\left\lfloor\frac{n-r}{2}\right\rfloor
            =\left\lceil\frac{n+r}{2}\right\rceil+\pa(n)+2-2s\\
\geq    & \left\lceil\frac{n}{2}\right\rceil
            +\left\lceil\frac{r}{2}\right\rceil
            +2-2s\geq\left\lceil\frac{r}{2}\right\rceil.
\end{align*}

Now we prove the claim.

(1) If $H$ contains only one or two vertices, then the assertion is
trivially true. So we assume that $\nu(H)\geq 3$. Let $u'$ be a
vertex in $H$ such that $d_H(u')$ is as small as possible. Thus
$d_H(v)\geq\lceil\lceil r/2\rceil/2\rceil$ for any vertex $v'\in
V(H)\backslash\{u'\}$. By Lemma 2, $H$ contains a path from $u$ of
order at least $\min\{\nu(H),\lceil r/2\rceil\}$.

(2) Let $B$ be the end-block of $H$ containing $D$ and $b$ be the
cut-vertex of $H$ contained in $B$. If $B$ contains only two
vertices, then the assertion is trivially true. So we assume that
$\nu(D)\geq 3$, from which we can see that $B$ is 2-connected. Let
$u'$ be a vertex in $B-b$ such that $d_{H}(u)$ is as small as
possible. Thus every vertex in $V(B)\backslash\{b,u'\}$ has degree
at least $\lceil\lceil r/2\rceil/2\rceil$ in $B$. By Lemma 2, $B$
contains a path from $b$ of order at least $\min\{\nu(B),\lceil
r/2\rceil\}$, and $H$ contains a path from $u$ of order at least
$\min\{\nu(B),\lceil r/2\rceil\}=\min\{\nu(D)+1,\lceil r/2\rceil\}$.
\end{proof}

Now we choose $D$ among all the non-separable components and
branches of $G-C$ such that the order of $D$ is as small as
possible. We set a parameter $a$ such that $a=0$ if $D$ is a
non-separable component, and $a=1$ if $D$ is a branch of $G-C$.

If $D$ is a branch of $G-C$, then let $H$ be the component of $G-C$,
and $B$ the end-block of $G-C$, containing $D$; if $D$ is a
component of $G-C$, then let $H=B=D$.

\begin{subcase}
  $\nu(D)=1$.
\end{subcase}

Let $v$ be the vertex in $D$. If $D=H$, then let $R=G-C-H$,
$X=N_C^+(H)$. If $D\neq H$, then let $y$ be a vertex in $H-B$,
$R=G-C-B-y$ and $X=N_C^+(H)\cup\{y\}$. Thus every component of $R$
is joined to at most one vertex in $X$. Moreover, we have
\begin{align*}
\nu(R)  & =\nu(G)-\nu(C)-1-2a\\
        & =m+n-s-n+r-1-2a=m+r-s-2a-1,
\end{align*}
and
\begin{align*}
|X|=d_C(H)+a\geq
d_C(v)+a=d(v)\geq\left\lceil\frac{n}{2}\right\rceil-s+1.
\end{align*}
Let $G'=G[V(R)\cup X]$. Note that there is a path of order at least
$2+2a$ with an end-vertex in $C$ and all other vertices in $H$. We
have $r\geq 2+2a$, and
\begin{align*}
\nu(G') & =\nu(R)+|X|\geq m+r-s-2a-1+\left\lceil\frac{n}{2}\right\rceil-s+1\\
        & =m+\left\lceil\frac{n}{2}\right\rceil+r-2s-2a
            \geq m+\left\lceil\frac{n}{2}\right\rceil+2-2s\geq m.
\end{align*}

\begin{claim}
  $D\neq H$ or $d_C(H)\geq 3$.
\end{claim}

\begin{proof}
Assume that $D=H$ and $d_C(H)=2$. Since $d_C(H)=d(v)\geq\lceil
n/2\rceil-s+1$, we have $n\leq 8$. We claim that every component of
$G-C$ is an isolated vertex. Suppose on the contrary that there is a
component $H'$ of $G-C$ with order at least 2. Note that there are
at least two better segments of $C$ with respect to $H'$. We have
$\nu(C)\geq 6$, and $G[V(C)\cup V(H')]$ contains a $P_8$, a
contradiction. Thus as we claimed, every component of $G-C$ is an
isolated vertex.

Note that $\nu(R)=m+r-s-1$. Since $s\leq 3$ (when $n\leq 8$) and
$r\geq 2$, we have $\nu(R)\geq m-2$. If $\nu(R)\geq m$, then there
is a $C_m$ in $\overline{G'}$; if $\nu(R)=m-1$, then $r=s\leq 3$,
and one of the two vertices in $N_C^+(H)$ is nonadjacent to every
vertex in $R$, and there is a $C_m$ in $\overline{G'}$; if
$\nu(R)=m-2$, then $r=s-1\leq 2$, and the two vertices in $N_C^+(H)$
are nonadjacent to every vertex in $R$, and there is a $C_m$ in
$\overline{G'}$. In any case we get a contradiction. So we conclude
that $D\neq H$ or $d_C(H)\geq 3$.
\end{proof}

By Claim 8, we can see that $|X|\geq 3$.

If there is a cycle $C'$ in $R$ with order $r+\pa(r)$, then let $P$
be a path between $C$ and $C'$, and $C\cup P\cup C'$ will contain a
$P_n$, a contradiction. Thus we assume that $R$ contains no cycle of
order $r+\pa(r)$. Since
\begin{align*}
        & \nu(R)+1-\frac{3}{2}(r+\pa(r))=m+r-s-2a-1+1-\frac{3}{2}(r+\pa(r))\\
\geq    & m-s-2a-\left\lceil\frac{r}{2}\right\rceil-\pa(r)
            \geq 2n-2s-1\geq 0,
\end{align*}
by Lemma 4, there is a path in $\overline{R}$ of order at least
\begin{align*}
p   & =\nu(R)+1-\frac{r+\pa(r)}{2}=m+r-s-2a-1+1
        -\left\lceil\frac{r}{2}\right\rceil\\
    & =m+\left\lfloor\frac{r}{2}\right\rfloor-s-2a.
\end{align*}

Note that
\begin{align*}
    & p+2|X|-3\geq m+\left\lfloor\frac{r}{2}\right\rfloor-s-2a+
        2\left(\left\lceil\frac{n}{2}\right\rceil-s+1\right)-3\\
=   & m+n+\pa(n)+\left\lfloor\frac{r}{2}\right\rfloor-3s-2a-1
        \geq m+n+\pa(n)-3s-a.
\end{align*}
We can see that $p+2|X|-3\geq m$, when $n\geq 9$, unless $n=11$ or
$12$ and $a=1$. If $n\leq 8$, then noting that $|X|\geq 3$, we also
have
$$p+2|X|-3\geq m+\left\lfloor\frac{r}{2}\right\rfloor-s-2a+3\geq m-s+3\geq m.$$
By Lemma 7, $\overline{G'}$ contains a $C_m$, a contradiction.

\noindent\textbf{Petty Case.} $n=11$ or $12$ and $a=1$.

We claim that every component of $G-C$ is a $K_1$, $K_2$, $K_3$ or a
star $K_{1,k}$. Suppose the contrary that there is a component $H'$
of order at least 4 which is not a star. Since there are at least
two best segments of $C$ with respect to $H'$, we can see that
$\nu(C)\geq 8$. Note that there is a path $P$ of order at least 5
with one end-vertex in $C$ and all other vertices in $H'$. This
implies that $G[V(C)\cup V(H')]$ contains a $P_{12}$, a
contradiction. Thus as we claimed, every component of $G-C$ is a
$K_1$, $K_2$, $K_3$ or a star $K_{1,k}$.

Since $H$ is not a $K_1$, $K_2$ or $K_3$, we conclude that $H$ is a
star. Now we choose a component $H'$ of $G-C$ that is a maximum star
of $G-C$, and let $u'$ be the center of $H'$, $v'$ and $y'$ be two
end-vertices of $H'$. Let $R'=G-C-\{u',v',y'\}$,
$X'=N_C^+(H')\cup\{y'\}$ and $G''=G[V(R')\cup X']$. By the analysis
above, we have
\begin{align*}
\nu(R')\geq m+r-s-3 \mbox{ and }
|X'|\geq\left\lceil\frac{n}{2}\right\rceil-s+1.
\end{align*}

Since $\nu(R')\geq m+r-s-3\geq 2n+2-s\geq 20$. If $G-C$ has at least
three components, then $R'$ is disconnected; if $G-C$ has exactly
two components, then $H'$ is a star with at least 4 vertices, and
$R'$ is connected; if $G-C$ consists of only one component $H'$,
then $R'=H'-\{u',v',y'\}$ is empty, and thus disconnected. Thus in
any case, $\overline{R'}$ is connected.

Let $H''$ be a component of $R'$ with the maximum order. If
$\nu(H'')\leq\lceil\nu(R')/2\rceil$, then every vertex of $R'$ has
degree at least $\lfloor\nu(R')/2\rfloor$ in $\overline{R'}$ . By
Lemma 2, $R'$ contains a Hamilton path. If
$\nu(H'')\geq\lceil\nu(R')/2\rceil+1$, then $H''$ is a star with at
least 4 vertices. Let $u''$ be the center of the star. Then every
vertex in $V(R')\backslash\{u''\}$ has degree at least
$\lceil\nu(R')/2\rceil$ in $\overline{R'-u''}$. By Lemma 2,
$\overline{R'-u''}$ contains a Hamilton cycle and $\overline{R'}$
contains a Hamilton path. In any case $R'$ contains a path of order
at least $p'=\nu(R')$. Thus we have
$$p'+2|X'|-3\geq\nu(R')+|X'|\geq m.$$

By Lemma 7, $\overline{G''}$ contains a $C_m$, a contradiction.

\begin{subcase}
  $\nu(D)=2$.
\end{subcase}

Let $v,v'$ be the two vertices in $D$. If $D=H$, then let $R=G-C-H$,
$X_1=N_C^+(H)$, $X_2=N_C^-(H)$. If $D\neq H$, then let $y$ be a
vertex in $H-B$, let $R=G-C-B-y$, $X_1=N_C^+(H)\cup\{y\}$,
$X_2=N_C^-(H)\cup\{y\}$. Thus every component of $R$ is joined to at
most one vertex in $X_i$, $i=1,2$, and
\begin{align*}
\nu(R)  & =\nu(G)-\nu(C)-2-2a\\
        & =m+n-s-n+r-2-2a\\
        & =m+r-s-2a-2.
\end{align*}
Let $X=X_1\cup X_2$ and $G'=G[V(R)\cup X]$. Note that there is a
path of order at least $3+3a$ with an end-vertex in $C$ and all
other vertices in $H$. We have that $r\geq 3+3a$.

Let $N_C(H)=\{z_1,z_2,\ldots,z_k\}$, where $z_i$, $1\leq i\leq k$,
are in order along $C$. Since there are at least two better
segments, we have $|X_1\backslash X_2|=|X_2\backslash X_1|\geq 2$.
For any vertex $z_i\in N_C(H)$: if $z_i$ is adjacent to exactly one
vertex in $\{v,v'\}$, then $\overrightarrow{C}[z_i,z_{i+1}]$ is a
good segment; if $z_i$ is adjacent to both $v$ and $v'$, then
$\overrightarrow{C}[z_i,z_{i+1}]$ is a better segment. This implies
that
\begin{align*}
|X|  & \geq d_C(v)+d_C(v')+a\\
        & \geq 2\left(\left\lceil\frac{n}{2}\right\rceil-s+1-1-a\right)+a\\
        & =n+\pa(n)-2s-a,
\end{align*}
and
\begin{align*}
\nu(G') & =\nu(R)+|X|\\
        & \geq m+r-s-2a-2+n+\pa(n)-2s-a\\
        & \geq m+3+3a-s-2a-2+n+\pa(n)-2s-a\\
        & \geq m+n+\pa(n)+1-3s\geq m.
\end{align*}

Since there are at least two better segments of $C$ with respect to
$H$, $\nu(C)\geq 6$. Thus there is a path in $G[V(C)\cup V(H)]$ of
order at least 8, which implies that $n\geq 9$.

\begin{claim}
  $D\neq H$ or $d_C(H)\geq 3$.
\end{claim}

\begin{proof}
Assume that $D=H$ and $d_C(H)=2$. Note that the two segments of $C$
with respect to $H$ are both better. Since $d_C(H)\geq
d(v)-1\geq\lceil n/2\rceil-s$, we have $n\leq 12$. We claim that
every component of $R$ has order at most $3$. Suppose on the
contrary that there is a component $H'$ of $G-C$ that has order at
least 4. Note that $H'$ is not a star. There are at least two best
segments of $C$ with respect to $H'$, which implies that $\nu(C)\geq
8$. By Claim 7, there is a path of order at least 5 with one
end-vertex in $C$ and all other vertices in $H'$. Thus $G[V(C)\cup
V(H')]$ contains a $P_{12}$, a contradiction. Thus as we claimed,
every component of $R$ has order 2 or 3.

Note that $\nu(R)=m+r-s-2$. Since $s\leq 4$ (when $n\leq 12$) and
$r\geq 3$, we have $\nu(R)\geq m-3$. If $\nu(R)\geq m$, then by
Lemma 3 there is a $C_m$ in $\overline{G'}$; if $\nu(R)=m-1$ or
$m-2$, then we have $r\leq s+1\leq 5$, and one of the two vertices
in $N_C^+(H)$ ($N_C^-(H)$) is nonadjacent to every vertex in $R$,
and there is a $C_m$ in $\overline{G'}$; if $\nu(R)=m-3$, then
$r=s-1\leq 3$, and every vertex in $N_C^+(H)$ and $N_C^-(H)$ is
nonadjacent to every vertex in $R$, and there is a $C_m$ in
$\overline{G'}$. In any case, we get a contradiction. So we conclude
that $D\neq H$ or $d_C(H)\geq 3$.
\end{proof}

By Claim 9, we have $|X_1|=|X_2|\geq 3$.

If there is a cycle in $R$ of order $r+\pa(r)$, then there will be a
path of order at least $n$ in $G$. Thus we assume that $R$ contains
no cycle of order $r+\pa(r)$. Since
\begin{align*}
        & \nu(R)+1-\frac{3}{2}(r+\pa(r))=m+r-s-2-2a+1-\frac{3}{2}(r+\pa(r))\\
\geq    & m-s-2a-\left\lceil\frac{r}{2}\right\rceil-\pa(r)-1
            \geq 2n-2s-2\geq 0,
\end{align*}
by Lemma 4, there is a path in $\overline{R}$ of order at least
\begin{align*}
p   & =\nu(R)+1-\frac{r+\pa(r)}{2}=m+r-s-2-2a+1
        -\left\lceil\frac{r}{2}\right\rceil\\
    & =m+\left\lfloor\frac{r}{2}\right\rfloor-s-2a-1.
\end{align*}

Note that
\begin{align*}
        & p+2|X|-5\\
\geq    & m+\left\lfloor\frac{r}{2}\right\rfloor-s-2a-1+2(n+\pa(n)-2s-a)-5\\
=       & m+2n+2\pa(n)+\left\lfloor\frac{r}{2}\right\rfloor-5s-4a-6\\
\geq    & m+2n+2\pa(n)-5s-7.
\end{align*}
We can see that $p+2|X|-5\geq m$, when $n\geq 13$. If $n\leq 12$,
then noting that $d_C(H)+a\geq 3$ and $|X|\geq 5$, we also have
$$p+2|X|-5\geq m+\left\lfloor\frac{r}{2}\right\rfloor-s-2a-1+5
\geq m-s+5\geq m.$$

By Lemma 8, $\overline{G'}$ contains a $C_m$, a contradiction.

\begin{subcase}
  $3\leq\nu(D)\leq\lceil r/2\rceil-1$.
\end{subcase}

In this case, $r\geq 7$. If $D=H$, then let $R=G-C-H$,
$X_1=N_C^+(H)$ and $X_2=N_C^-(H)$. If $D\neq H$, then let $y$ be a
vertex in $H-B$ which is not a cut-vertex of $H-B$, let $R=G-C-B-y$,
$X_1=N_C^+(H)\cup\{y\}$ and $X_2=N_C^-(H)\cup\{y\}$. Thus every
component of $R$ is joined to at most one vertex in $X_i$, $i=1,2$,
and
\begin{align*}
\nu(R)  & =\nu(G)-\nu(C)-\nu(D)-2a\\
        & \geq m+n-s-n+r-\left\lceil\frac{r}{2}\right\rceil+1-2a\\
        & =m+\left\lfloor\frac{r}{2}\right\rfloor+1-s-2a.
\end{align*}
Clearly, every component of $R$ has order at least 2, and
\begin{align*}
        & \left\lceil\frac{3\nu(R)}{2}\right\rceil+4
            \geq\left\lceil\frac{3(m+\lfloor r/2\rfloor+1-s
            -2a)}{2}\right\rceil+4\\
\geq    & m+\left\lceil\frac{m+3(3+1-s-2a)}{2}\right\rceil+4\\
\geq    & m+\left\lceil\frac{2n+15-3s}{2}\right\rceil\geq m.
\end{align*}

Let $X=X_1\cup X_2$ and $G'=G[V(G-B)\backslash N_C(H)]$. Since there
are at least two best segments with respect to $H$, we have
$|X_1\backslash X_2|=|X_2\backslash X_1|\geq 2$. Let $v$ be a vertex
in $D$.

Since $R$ contains no cycle of length $r+\pa(r)$ and
\begin{align*}
        & \nu(R)+1-\frac{3}{2}(r+\pa(r))\\
=       & m+\left\lfloor\frac{r}{2}\right\rfloor+1-s-2a+1
            -3\cdot\left\lceil\frac{r}{2}\right\rceil\\
\geq    & m+2-r-2\pa(r)-s-2a\\
\geq    & 2n+3-2s+\pa(n)+2-2\pa(r)-s-2a\\
\geq    & 2n+\pa(n)+1-3s\geq 0,
\end{align*}
$\overline{R}$ contains a path of order at least
\begin{align*}
p   & =\nu(R)+1-\frac{r+\pa(r)}{2}\\
    & \geq m+\left\lfloor\frac{r}{2}\right\rfloor+1-s-2a+1
        -\left\lceil\frac{r}{2}\right\rceil\\
    & =m+2-s-\pa(r)-2a.
\end{align*}

\begin{claim}
  $D\neq H$ or $d_C(H)\geq 3$.
\end{claim}

\begin{proof}
Assume that $D=H$ and $d_C(H)=2$. Thus
$$d_D(v)\geq\left\lceil\frac{n}{2}\right\rceil-s+1-2
=\left\lceil\frac{n}{2}\right\rceil-s-1,$$
and
$$\nu(D)\geq 1+d_D(v)\geq\left\lceil\frac{n}{2}\right\rceil-s
\geq s-2\geq\left\lceil\frac{r}{2}\right\rceil-1.$$ This implies
that $\lceil r/2\rceil=s-1$, $\nu(D)=\lceil r/2\rceil-1$, and
$d_D(v)=\nu(D)-1$. Note that in this case every vertex in $D$ has
degree $\nu(D)-1$, and thus $D$ is a clique.

If every vertex in $N_C^+(H)$ is joined to some component of $G-C$,
then by Claim 7, we can find a path from the cycle $C$, component
$H$ and the two components joined to the two vertices in $N_C^+(H)$,
of order at least
\begin{align*}
    & \nu(C)+3\nu(D)=\nu(C)+3\cdot\left(\left\lceil\frac{r}{2}\right\rceil-1\right)\\
=   & n-r+r+\pa(r)+\left\lceil\frac{r}{2}\right\rceil-3\geq n,
\end{align*}
a contradiction. Thus there is a vertex $v'$ in $N_C^+(H)$ that is
not joined to every component of $G-C$. Let $G''=G-C$.

Since $\lceil r/2\rceil=s-1$ and $r\geq 7$, we can see that $r\geq
s+1$. Thus
$$\nu(G'')=\nu(G)-\nu(C)\geq m+n-s-n+r=m+r-s\geq m.$$
Note that in this case, $\overline{G''-H}=\overline{R}$ contains a
path of order at least $p\geq m+2-s-\pa(r)\geq m+3-r-\pa(r)$ and
\begin{align*}
    & p+2\nu(H)-1\geq m+3-r-\pa(r)
        +2\cdot\left(\left\lceil\frac{r}{2}\right\rceil-1\right)-1\\
=   & m+3-r-\pa(r)+r+\pa(r)-2-1=m.
\end{align*}
Since
$$\nu(R)\geq m+\left\lfloor\frac{r}{2}\right\rfloor+1-s=m-\pa(r)
\geq\left\lceil\frac{m}{2}\right\rceil,$$ by Lemma 5,
$\overline{G''}$ contains a $C_m$, and $\overline{G}$ contains a
$W_m$ with the hub $v'$, a contradiction.
\end{proof}

By Claim 10, $|X_1|=|X_2|\geq 3$. If $D\neq H$, then since there are
at least two best segments with respect to $H$, we can see that
$\nu(C)\geq 8$. By Claim 7, there is a path of order at least 6 with
one end-vertex in $C$ and all other vertices in $H$, which implies
that $G[V(C)\cup V(H)]$ contains a $P_{13}$; if $D=H$ and
$d_C(H)\geq 3$, noting that at least two segments of $C$ with
respect to $H$ are best, we have $\nu(C)\geq 10$. By Claim 7, there
is a path of order at least 4 with an end-vertex in $C$ and all
internal vertices in $H$, $G[V(C)\cup V(H)]$ contains a $P_{13}$ as
well. Thus we conclude that $n\geq 14$.

Let $H'$ be a component of $R$, and let $W$ be the union of $X$ and
the set of vertices in $V(C)\backslash N_C(H)$ not joined to $H'$.
For any two vertices $x,y$ with $xy\in E(C)$: if one of $x,y$ is in
$N_C(H)$, then the other one will be in $X\subset W$; if none of
them is in $N_C(H)$, then at least one of them will not be joined to
$H'$, otherwise there will be a cycle longer than $C$. This implies
that $|W|\geq \lceil(n-r)/2\rceil+a=q$.

Since
\begin{align*}
        & \nu(R)+q-1\\
\geq    & m+\left\lfloor\frac{r}{2}\right\rfloor+1-s-2a
            +\left\lceil\frac{n-r}{2}\right\rceil+a-1\\
\geq    & m+\left\lfloor\frac{n}{2}\right\rfloor-s-a\geq m,
\end{align*}
($n\geq 14$) and
\begin{align*}
        & p+2q-5\\
\geq    & m+2-s-\pa(r)-2a
            +2\cdot\left(\left\lceil\frac{n-r}{2}\right\rceil+a\right)-5\\
=       & m+n-r-s-3+\pa(n-r)-\pa(r)\\
=       & m+n-3s-1+\pa(n)+\pa(n-r)-\pa(r)\\
\geq    & m+n-3s-1,
\end{align*}
we can see that $p+2q-5\geq m$, unless $n=15$, $s=5$ and $r=7$.

\noindent\textbf{Petty Case.} $n=15$, $s=5$ and $r=7$.

In this case, $\nu(C)=8$ which implies that $D\neq H$. It is easy to
find a path with two end-vertices in $C$ and all internal vertices
in $H$ of order at least 7. Thus $\nu(C)\geq 12$, a contradiction.

By Lemma 9, $\overline{G'}$ contains a $C_m$, a contradiction.

\begin{subcase}
  $\nu(D)\geq\max\{\lceil r/2\rceil,3\}$.
\end{subcase}

By Claim 7, there is a path of order at least 4 with an end-vertex
in $C$ and all other vertices in $H$. Thus we have $r\geq 4$. Let
$H'$ be an arbitrary component of $R$ and $u\in V(H')$. By Claim 7,
$H'$ contains a path from $u$ of order at least $\lceil r/2\rceil$.
Thus for any edge $xy\in E(C)$, either $x$ or $y$ is not joined to
any components of $G-C$, otherwise there will be a $P_n$ in $G$.
Moreover, if $r$ is odd and $x$ is joined to some component, say
$H'$, of $G-C$, then $x^{++}$ will not be joined to any component of
$G-C$ other than $H'$ as well.

\begin{subsubcase}
Every component of $G-C$ has order at most $r-1$.
\end{subsubcase}

Let $v$ be a vertex in $N_C^+(H)$, and let $G'=G[V(G-C)\cup
N_C^+(H)\backslash\{v\}]$. Note that $v$ is nonadjacent to every
vertex in $G'$, and every component of $G'$ has order at most
$$r-1\leq 2s-\pa(n)-2-1\leq\left\lceil\frac{n}{2}\right\rceil
\leq\left\lfloor\frac{m}{2}\right\rfloor.$$
Let $u$ be a vertex in $H$. Since
$$d_C(H)\geq d_C(u)\geq d(u)-\nu(H)+1\geq d(u)+2-r,$$
and
\begin{align*}
\nu(G') & =\nu(G)-\nu(C)+d_C(H)-1\\
        & \geq m+n-s-n+r+d(v)+1-r\\
        & \geq m-s+\left\lceil\frac{n}{2}\right\rceil-s+1+1\\
        & = m+\left\lceil\frac{n}{2}\right\rceil+2-2s\geq m,
\end{align*}
by Lemma 3, there is a $C_m$ in $\overline{G'}$, a contradiction.

\begin{subsubcase}
There is a component of $G-C$ of order at least $r$.
\end{subsubcase}

Let $H'$ be a component of $G-C$ with order at least $r$. We claim
that there is a vertex $u$ in $H'$ with $d_{H'}(u)\leq\lceil
r/2\rceil-1$. Suppose the contrary that every vertex of $H'$ has
degree at least $\lceil r/2\rceil$ in $H'$. If $H'$ is 2-connected,
then by Lemma 2, there is a cycle of order at least $r$ in $H'$, and
$G$ will contain a $P_n$; if $G$ is separable, letting $B'$ be any
end-block of $H'$, $b'$ be the cut-vertex of $H'$ contained in $B'$,
and $u'$ be any vertex in $V(B)\backslash\{b'\}$, then there is a
path from $b'$ to $u'$ of order at least $\lceil r/2\rceil+1$. Thus
$G$ will contain a $P_n$ as well. So we assume that there is a
vertex $u$ in $H'$ with $d_{H'}(u)\leq\lceil r/2\rceil-1$.

Let $v$ be a vertex in $N_C^+(H')$, $X=N_C^+(H')\backslash\{v\}$. If
$r$ is odd, then let $\overrightarrow{C}[z,z']$ be a better segment
of $C$ with respect to $H'$ not containing $v$, and we add $z^{++}$
to $X$. Let $G'=G[V(G-C)\cup X]$. Note that $v$ is nonadjacent to
every vertex in $G'$, and there are no edges between $G-C$ and $X$.

Since
\begin{align*}
        & d_C(H')\geq d_C(u)=d(u)-d_{H'}(u)\\
\geq    &\left\lceil\frac{n}{2}\right\rceil-s+1
            -\left\lceil\frac{r}{2}\right\rceil+1
            =\left\lceil\frac{n}{2}\right\rceil+2
            -\left\lceil\frac{r}{2}\right\rceil-s,
\end{align*}
we have
$$|X|=d_C(H)-1+\pa(r)\geq\left\lceil\frac{n}{2}\right\rceil+1
-\left\lfloor\frac{r}{2}\right\rfloor-s$$
and
\begin{align*}
\nu(G') & =\nu(G)-\nu(C)+|X|\\
        & \geq m+n-s-n+r+\left\lceil\frac{n}{2}\right\rceil+1
            -\left\lfloor\frac{r}{2}\right\rfloor-s\\
        & \geq m-s+\left\lceil\frac{r}{2}\right\rceil
            +\left\lceil\frac{n}{2}\right\rceil-s+1\\
        & \geq m+\left\lceil\frac{n}{2}\right\rceil+3-2s\geq m.
\end{align*}

Since $G-C$ contains no cycle of length $r+\pa(r)$ and
\begin{align*}
        & \nu(G-C)+1-\frac{3}{2}(r+\pa(r))\\
=       & m+r-s+1-3\cdot\left\lceil\frac{r}{2}\right\rceil\\
\geq    & m-\left\lceil\frac{r}{2}\right\rceil-\pa(r)-s\\
\geq    & 2n-2s\geq 0,
\end{align*}
$\overline{G-C}$ contains a path of order at least
\begin{align*}
p   & =\nu(G-C)+1-\frac{r+\pa(r)}{2}\\
    & =m+r-s+1-\left\lceil\frac{r}{2}\right\rceil
        =m+\left\lfloor\frac{r}{2}\right\rfloor+1-s.
\end{align*}
Clearly $|X|\geq 1$. If $\lfloor r/2\rfloor\geq s-2$, then
\begin{align*}
        & p+2|X|-1\\
\geq    & m+\left\lfloor\frac{r}{2}\right\rfloor+1-s+2-1\\
\geq    & m+s-2+1-s+2-1=m.
\end{align*}
If $\lfloor r/2\rfloor\leq s-3$, then
\begin{align*}
        & p+2|X|-1\\
\geq    &  m+\left\lfloor\frac{r}{2}\right\rfloor+1-s+
            2\cdot\left(\left\lceil\frac{n}{2}\right\rceil+1
            -\left\lfloor\frac{r}{2}\right\rfloor-s\right)-1\\
=       & m+n+\pa(n)+2-\left\lfloor\frac{r}{2}\right\rfloor-3s\\
\geq    & m+n+\pa(n)+5-4s\geq m.
\end{align*}
Since
$$\nu(G-C)=m+r-s\geq\left\lceil\frac{m}{2}\right\rceil,$$ by Lemma
5, there is a $C_m$ in $\overline{G'}$, a contradiction.

The proof is complete.\hfill$\Box$

\section{Remarks}

A \emph{linear forest} is a forest such that every component of it
is a path. From our main result of the paper, we can conclude the
following result.

\begin{corollary}
Let $n\geq 2$, $m\geq 2n+1$ and $F$ be a linear forest on $m$
vertices. Then $$R(P_n,K_1\vee F)=t(n,m).$$
\end{corollary}

\begin{proof}
Note that the graph constructed in the beginning of Section 3
contains no $P_n$ and its complement contains no $K_1\vee F$. We
conclude that $R(P_n,K_1\vee F)\geq t(n,m)$. On the other hand,
since $K_1\vee F$ is a subgraph of $W_m$, we have $R(P_n,K_1\vee
F)\leq R(P_n,W_m)\leq t(n,m)$.
\end{proof}

For the case $F$ is an empty graph, the above formula gives the
Ramsey numbers of paths versus stars when $m\geq 2n+1$. In fact,
Parsons \cite{Parsons} gave all the values of the path-star Ramsey
numbers by a recursive formula. The interested readers can compare
our formula with the recursive one in \cite{Parsons}.

\section*{Acknowledgments}

The authors are very grateful to Professor Yunqing Zhang for
providing them the paper \cite{Zhang}.


\begin{thebibliography}{10}
\bibitem{Bakoro_Surahmat}
E.T. Baskoro and Surahmat, The Ramsey number of paths with respect
to wheels, \emph{Discrete Math.} {\bf 294} (2005) 275--277.

\bibitem{Bondy_Murty}
J.A.~Bondy and U.S.R.~Murty, Graph Theory with Applications,
Macmillan, London and Elsevier, New York, 1976.

\bibitem{Chen_Zhang_Zhang}
Y. Chen, Y. Zhang and K. Zhang, The Ramsey numbers of paths versus
wheels, \emph{Discrete Math.} {\bf 290} (2005) 85--87.

\bibitem{Dirac}
G.A. Dirac, Some theorems on abstract graphs, \emph{Proc. London.
Math. Soc.} {\bf 2} (1952) 69--81.

\bibitem{Erdos}
P. Erd\"{o}s, Some remarks on the theory of graphs, \emph{Bull.
Amer. Math. Soc.} {\bf 53} (1947) 292--294.

\bibitem{Erdos_Gallai}
P. Erd\"{o}s and T. Gallai, On maximal paths and circuits of graphs,
\emph{Acta Math. Acad. Sci. Hungar.} {\bf 10} (1959) 337--356.

\bibitem{Fan}
G.~Fan, New sufficient conditions for cycles in graphs, \emph{J.
Combin. Theory, Ser. B} {\bf 37} (1984) 221--227.

\bibitem{Faudree_Lawrence_Parsons_Schelp}
R.J. Faudree, S.L. Lawrence, T.D. Parsons and R.H. Schelp,
Path-cycle Ramsey numbers, \emph{Discrete Math.} {\bf 10} (1974)
269--277.

\bibitem{Gerencser_Gyarfas}
L. Gerencs\'{e}r and A. Gy\'{a}rf\'as, On Ramsey-type problems,
\emph{Ann. Univ. Sci. Budapest. E\"{o}tv\"{o}s Sect. Math.} {\bf 10}
(1967) 167--170.

\bibitem{Parsons}
T.D. Parsons, Path-star Ramsey numbers, \emph{J. Combin. Theory,
Ser. B} {\bf 17} (1974) 51--58.

\bibitem{Salmana_Broersma}
A.N.M. Salmana and H.J. Broersma, On Ramsey numbers for paths versus
wheels, \emph{Discrete Math.} {\bf 307} (2007) 975--982.

\bibitem{Zhang}
Y. Zhang, On Ramsey numbers of short paths versus large wheels,
\emph{Ars Combin.} {\bf 89} (2008) 11--20.

\end{thebibliography}
\end{document}